\documentclass[a4paper, 12pt, twoside, notitlepage]{amsart}

\usepackage{amsmath,amscd}
\usepackage{amssymb}
\usepackage{amsthm}
\usepackage{comment}
\usepackage{graphicx, xcolor}
\usepackage{mathrsfs}
\usepackage[ocgcolorlinks, linkcolor=blue]{hyperref}

\usepackage{bm}
\usepackage{bbm}
\usepackage{url}

\usepackage[utf8]{inputenc}
\usepackage{mathtools,amssymb}
\usepackage{esint}
\usepackage{tikz}
\usepackage{dsfont}
\usepackage{relsize}
\usepackage{url}
\usepackage[shortlabels]{enumitem}
\usepackage{lineno}
\usepackage{enumitem}
\usepackage{fullpage} % wide page
\usepackage{verbatim}
\usepackage{dsfont}
\numberwithin{equation}{section}

\allowdisplaybreaks

\mathtoolsset{showonlyrefs}

\graphicspath{{images/}}

\newtheorem{theorem}{Theorem}[section]
\newtheorem{lemma}[theorem]{Lemma}
\newtheorem{definition}[theorem]{Definition}
\newtheorem{corollary}[theorem]{Corollary}

\newtheorem{proposition}[theorem]{Proposition}

\newtheorem{remark}[theorem]{Remark}

\title[Broken ray transform for twisted geodesics]{Broken ray transform for twisted geodesics on surfaces with a reflecting obstacle}
		\author[S.R. Jathar]{Shubham R. Jathar}
\address{ Indian Institute of Science Education and Research (IISER) Bhopal, India \& Computational Engineering, School of Engineering Sciences,
Lappeenranta-Lahti University of Technology LUT, Lappeenranta, Finland}
\email {shubham.jathar@lut.fi}

\author[M. Kar]{Manas Kar}
\address{ Indian Institute of Science Education and Research (IISER) Bhopal, India}
\email{manas@iiserb.ac.in}
\author[J. Railo]{Jesse Railo}
\address{ Department of Pure Mathematics and Mathematical Statistics,
University of Cambridge,
Cambridge CB3 0WB, UK \& Computational Engineering, School of Engineering Sciences,
Lappeenranta-Lahti University of Technology LUT, Lappeenranta, Finland}
\email{jesse.railo@lut.fi}

\newcommand{\R}{{\mathbb R}}
\newcommand{\Z}{{\mathbb Z}}

\newcommand{\id}{\mathrm{Id}}

\newcommand{\E}{\mathcal E}
\newcommand{\rR}{\mathcal R}

\NewDocumentCommand{\sff}{}{\mathrm{I\!I}}
\newcommand{\norm}[1]{\lVert #1 \rVert}
\newcommand{\abs}[1]{\left\lvert #1 \right\rvert}%absolute value
\newcommand{\ip}[2]{\left\langle #1,#2 \right\rangle}%inner product or duality pairing
\DeclareMathOperator{\inte}{int} %divergence

\begin{document}
	
	\maketitle
	\begin{abstract} 
		We prove a uniqueness result for the broken ray transform acting on the sums of functions and $1$-forms on surfaces in the presence of an external force and a reflecting obstacle. We assume that the considered twisted geodesic flows have nonpositive curvature. The broken rays are generated from the twisted geodesic flows by the law of reflection on the boundary of a suitably convex obstacle. Our work generalizes recent results for the broken geodesic ray transform on surfaces to more general families of curves including the magnetic flows and Gaussian thermostats.
		\medskip
		
		
  %44A12  	Radon transform
  %58J90  	Applications of PDEs on manifolds
  %37E35  	Flows on surfaces

	\end{abstract}

	\tableofcontents

\section{Introduction}

This article studies generalizations of the geodesic ray transform to general families of curves. Our main focus will be in broken ray tomography where the trajectories of particles may reflect from the boundary of a reflecting obstacle according to the law of reflection. Furthermore, we consider a situation where the trajectories are influenced by an external force such a magnetic field. Our study limits to the two dimensional case. Our main result, stated later in Theorem \ref{thm: main theorem 1}, shows that a function is uniquely determined from the collection of all of its line integrals over the twisted broken rays. We also obtain an analogous result corresponding to vector field tomography with its natural gauge.

Let $(M,g)$ be a smooth Riemannian surface with smooth boundary and $\lambda \in C^\infty(SM)$. We say that a curve $\gamma$ is a twisted geodesics if it satisfies the $\lambda$-geodesic equation
\begin{equation}\label{eq:lambda-geodesic-eq:intro}
D_t\dot{\gamma}=\lambda(\gamma,\dot{\gamma})i\dot{\gamma}
\end{equation}
where $i$ is the rotation by $90$ degrees counterclockwise.
The term $\lambda(\gamma,\dot{\gamma})i\dot{\gamma}$ represents an external force that pushes a particle out from its usual geodesic trajectory where a particle without any influence from an external force would continue its motion. The case when $\lambda(x,v)$ does not depend on the vertical variable $v$ is called magnetic and the case where $\lambda(x,v)$ is linear in $v$ is called Gaussian thermostatic, both widely studied in dynamics \cite{Dairbekov:Paternain:2007:Entropy1,Dairbekov:Paternain:2007:II,Mettler-Paternain-vortices-2022, Smale67-differentiable-dynamical-systems-thermostats,Wojtkowski:2000,Wojtkowski:2000b}.

The geodesic ray transform and closely related Radon transforms are studied by many authors and the area has a long history starting from the early twentieth century \cite{Helgason:2011:book,Ilmavirta:Monard:2019,GIP2D}. More recent advances include the solenoidal injectivity of tensor tomography in two dimensions \cite{Muhometov:1977,Paternain_Salo_Uhlmann2013,Sharafutdinov:2007}. In higher dimensions, the geodesic ray transforms are fairly well-understood in negative curvature \cite{Guillarmou:Paternain:Salo:Uhlmann:2016,Paternain-Salo-sharp-stability-negative-curvature,Paternain_Salo_Uhlmann2015}, and when a manifold has a strictly convex foliation \cite{deHoop-Uhlmann-Zhai-local-high-order-tensor-fields,Paternain:Salo:Uhlmann:Zhou:2019,Stefanov:Uhlmann:Vasy:2018,Uhlmann:Vasy:2016}. The geodesic ray transform is closely related to the boundary rigidity problem \cite{Pestov-Uhlmann-boundary-rigidity,Stefanov-Uhlmann-Vasy-boundary-rigidity} and the spectral rigidity of closed Riemannian manifolds \cite{Guillarmou-Lefeuvre-length-spectrum-rigidity,guillarmou-lefeuvre-paternain-anosov-length-spectrum-2023,PSU-spectral-rigidity-anosov-2d}. Other recent considerations include generalizations of many existing results to some classes of open Riemannian manifolds \cite{Eptaminitakis-Graham-assymptotically-conic,Graham-etal-assymptotically-hyperbolic,Guillarmou-Lassas-Tzou-assymptotically-conic,Lehtonen-Railo-Salo-Cartan-Hadamard} and to the matrix weighted ray transforms \cite{Ilmavirta-Railo-piecewise-matrix-weighted,Paternain:Salo:Uhlmann:Zhou:2019} as well as their statistical analysis \cite{Monard-Nickl-Paternain-nonabelian-raytransform,Monard-Nickl-Paternain-statistical-guarantees}.

The ray transforms for twisted geodesics and general families of curves have been studied recently in \cite{Assylbekov:Dairbekov:2018,Zhang:2023} and in the appendix of \cite{Uhlmann:Vasy:2016} by Hanming Zhou. It is shown in \cite{Assylbekov:Dairbekov:2018} that the twisted geodesic ray transform is injective on the simple Finslerian surfaces. For the most recent other studies, we refer to \cite{Zhang:2023}. In \cite{jathar2023loop}, the injectivity of the matrix attenuated and nonabelian ray transforms for the nontrapping magnetic and thermostatic flows on compact surfaces with boundary are studied. We give a more detailed account of the works that study inverse problems for the magnetic or Gaussian thermostat flows in Sections \ref{sec:magnetic} and \ref{sec:thermostat}, respectively. 

The broken ray transform has been studied extensively. In the case of a strictly convex obstacle, the uniqueness result for the broken ray transform of scalar functions on Riemannian surfaces of nonpositive curvature were obtained in \cite{Ilmavirta:Salo:2016}. This result was later generalized to higher dimensions and tensor fields of any order in \cite{Ilmavirta:Paternain:2022}. In the case of rotational (or spherical) symmetry, one may sometimes solve these and related problems using local results and data avoiding the obstacle when the manifold satisfies the Herglotz condition \cite{deHoop-Ilmavirta-Katsnelson-spherical-spectral-rigidity,Ilmavirta-Mönkkönen-spherical-finslerian-ray,Sharafutdinov:1997}. Broken lens rigidity was studied recently in \cite{deHoop-etall-broken-scattering-relation}, and a broken non-Abelian ray transform in Minkowski space in \cite{St-Amant-broken-non-abelian-Minkowski-space}. Other geometric results include boundary determination from a broken ray transform \cite{Ilmavirta:2014} and a reflection approach using strong symmetry assumptions \cite{Ilmavirta:2015}, for example letting to solve the broken ray transforms on flat boxes over closed billiard trajectories \cite{Ilmavirta-torus,Hubenthal-flat-reflection}. Numerical reconstruction algorithms and stability for the mentioned problem on the flat boxes would follow directly from \cite{Ilmavirta-Koskela-Railo-2020-torus-computed,Railo-fourier-periodic-radon}. Artifacts appearing in the inversion of a broken ray transform was studied recently in the flat geometry \cite{Zhang:2020}. We refer to the work \cite{Eskin-multiple-obstacles-2004} regarding a possibility to have more than just one reflecting obstacle. Finally, some related results without proofs are stated in the setting of curve families in the Euclidean disk \cite{Mukhometov-1991}.

\subsection{Main result}

We briefly recall the setting of our work, for further details we point to the later sections. Let $(M,g)$ be a compact oriented smooth Riemannian surface with smooth boundary and $\lambda \in C^\infty(SM)$. Assume that $\partial M = \E \cup \rR$ where $\E$ and $\rR$ are two relatively open disjoint subsets. Let $\nu$ denote the inward unit normal. We define the reflection map $\rho: \partial SM \rightarrow \partial SM$ by
\[
\rho(x, v)=\left(x, v-2\langle v, \nu\rangle_g \nu\right).
\]  A curve $\gamma$ on $M$ is called a \emph{broken $\lambda$-ray} if $\gamma$ is a $\lambda$-geodesic in $\inte(M)$ and reflects on $\mathcal{R}$ according to the law of reflection \begin{equation}\rho(\gamma(t_0-),\dot{\gamma}(t_0-))=(\gamma(t_0+),\dot{\gamma}(t_0+))\end{equation} whenever there is a reflection at $\gamma(t_0) \in \mathcal{R}$.

We call a dynamical system $(M,g,\lambda,\mathcal{E})$ \emph{admissible} (cf. Definition \ref{def:admis:lam:brok} with more details) if
\begin{enumerate}[(i)]
\item the emitter $\mathcal{E}$ is strictly $\lambda$-convex;
\item the obstacle $\mathcal{R}$ has admissible signed $\lambda$-curvature;
\item the Gaussian $\lambda$-curvature of $(M,g,\lambda)$ is nonpositive;
\item \label{eq:nontrapping-intro}the broken $\lambda$-geodesic flow is nontrapping;
\item there exists $a > 0$ such that every broken $\lambda$-ray $\gamma$ has at most one reflection at $\mathcal{R}$ with $\abs{\ip{\nu}{\dot{\gamma}}}<a$.
\end{enumerate}
\noindent
The \emph{broken $\lambda$-ray transform} of $f\in C^2(SM)$ is defined by
\begin{equation}\label{eq:integral_equation-intro}
If(x,v)=\int_0^{\tau_{x,v}}f(\phi_t(x,v))dt, \ \ \ (x,v) \in \pi^{-1}\E,
\end{equation}
where $\phi_t(x,v) = (\gamma_{x,v}(t), \dot\gamma_{x,v}(t))$ is the broken $\lambda$-geodesic flow, $\tau_{x,v}$ is the travel time of $\phi_t(x,v)$ and $\pi: SM \to M$ is the projection $\pi(x,v)=x$. Our main theorem is the following uniqueness result for the sums of functions and $1$-forms.
\begin{theorem}\label{thm: main theorem 1}
Let $(M,g,\lambda,\E)$ be an admissible dynamical system with $\lambda\in C^{\infty}(SM)$. Let $f(x,\xi)=f_0(x)+\alpha_j(x)\xi^j$ where $f_0 \in C^2(M)$ is a function and $\alpha$ is a $1$-form with coefficients in $C^2(M)$. If $If=0$, then $f_0=0$ and $\alpha=dh$ where $h \in C^3(M)$ is a function such that $h|_{\E}=0$.
\end{theorem}

Our proof is based on the ideas introduced in \cite{Ilmavirta:Paternain:2022,Ilmavirta:Salo:2016} and the Pestov identity for the twisted geodesic flows in \cite{Assylbekov:Dairbekov:2018,Dairbekov:Paternain:2007:Entropy1}. The proof could be split into three main parts, each of them having their own technical challenges, resolved in our work for a general $\lambda \in C^\infty(SM)$.
\begin{enumerate}[(i)]
\item One has to analyze a generalized Pestov identity with boundary terms and decompose the boundary terms into the even and odd parts with respect to the reflection;
\item One has to reduce the problem into a related transport problem for the broken $\lambda$-geodesics and observe if the Pestov identity and the analysis of the first step applies to the solutions of the transport equation, then the problem can be solved;
\item One has to show sufficient regularity for the solutions of the broken transport equation. This can be done by analyzing carefully the behavior of the broken $\lambda$-rays, broken $\lambda$-Jacobi fields at the reflection points, and utilizing a time-reversibility property for the pair of flows with respect to $\lambda(x,v)$ and $-\lambda(x,-v)$.
\end{enumerate}

\section{Preliminaries}
In this section, we introduce our notation and concisely present many primary definitions employed throughout the article.
We closely follow the notation used in the recently published book by Paternain, Salo and Uhlmann \cite{GIP2D}. We also refer to \cite{Lee:riem:2nd} for the basics of Riemannian geometry. 
\subsection{Basic notation}
Throughout the article, we denote by $(M,g)$ a complete oriented smooth Riemannian manifold with or without boundary. We always assume that $M$ is a surface, i.e. $\dim(M)=2$. We denote the Levi-Civita connection or covariant derivative of $g$ by $\nabla$, and the determinant of $g$ by $\abs{g}$. When the covariant derivative is restricted to a smooth curve $\gamma$, we simply write $D_t =\nabla_{\dot{\gamma}}$. We sometimes emphasize the base point $x \in M$ in the notation of the metric as $g_x$ and other operators but this is often omitted. We denote the volume form of $(M,g)$ by $dV^2 := \abs{g}^2 dx_1 \wedge dx_2$, expressed in any local positively oriented coordinates. For any vector $v \in T_xM$, let $v^\perp \in T_xM$ denote the unique vector obtained by rotating $v$ counterclockwise by $90^\circ$. This vector satisfies:
\[\left|v^{\perp}\right|_g=|v|_g, \quad\left\langle v, v^{\perp}\right\rangle=0\]
and forms a positively oriented basis of $T_xM$ with $v$, when $v \neq 0$. Note that, often, we may also write $v^\perp$ as $iv$. Given a vector $v \in T_xM$ and a positive orientation on $M$, we denote by $v_\bot = -v^\bot$ the clockwise rotation. We denote by $K$ the Gaussian curvature of $(M,g)$.

The signed curvature of the boundary $\partial M$ is defined as
\begin{equation}
\kappa := \langle D_t \dot\delta(t), \nu \rangle_g
\end{equation}
where $\delta(t)$ represents an oriented unit-speed curve that parametrizes the boundary $\partial M$ and $\nu$ is the inward unit normal along $\delta(t)$.
For a comprehensive explanation, please refer to
 \cite[Chapter 9, p. 273]{Lee:riem:2nd}. 
Furthermore, we define the second fundamental form of the boundary $\partial M$ as follows: \begin{equation}
    \sff_x(v,w) := -\ip{\nabla_v \nu}{w}_g
\end{equation}
where $x \in \partial M$ and $v,w \in T_x\partial M$ (cf. \cite[p. 56]{GIP2D}). We say that $\partial M$ is strictly convex at $x \in \partial M$ if $\sff_{x}$ is positive definite, i.e. $\sff_x(v,v) > 0$ for any $v \in T_x\partial M\setminus\{0\}$. We say that $M$ has a strictly convex boundary if $\partial M$ is strictly convex for any $x \in \partial M$. We say that $\partial M$ is strictly concave at $x \in \partial M$ if $\sff_{x}$ is negative definite, i.e. $\sff_x(v,v)< 0$ for any $v \in T_x\partial M\setminus\{0\}$. The relation between the signed curvature and the second fundamental form is given by $\kappa = \sff(\dot{\delta},\dot{\delta})$ (cf. \cite[Chapters 8 and 9]{Lee:riem:2nd}). If the boundary is strictly convex, it implies that the signed curvature is positive, whereas if it is strictly concave, the signed curvature is negative.

\subsection{Analysis on the sphere bundle}

 In this article, we often assume that $M$ is a Riemannian manifold with smooth boundary $\partial M$. We denote by $SM$ the unit sphere bundle of $M$, i.e.
\begin{equation}
    SM := \{\,(x,v) \,;\, v \in T_xM, \abs{v}_g=1\,\}.
\end{equation}
The boundary of $SM$ is given by
\begin{equation}\partial SM:= \{\,(x,v)\in SM\,;\, x \in \partial M\,\}.
\end{equation}
We define the influx and outflux boundaries of $SM$ as the following sets
\begin{equation}
    \partial_{\pm} SM :=\{\,(x,v) \in \partial SM \,;\, \pm \ip{v}{\nu(x)}_g\geq0\,\}.
\end{equation}
The glancing region is defined as $\partial_0SM := \partial_+SM \cap \partial_-SM = S(\partial M)$. We denote by $dS_x$ the volume form of $(S_xM,g_x)$ for any $x \in M$. The sphere bundle $SM$ is naturally associated with the measure
\begin{equation}
    d\Sigma^3 := dV^2 \wedge dS_x
\end{equation}
called the Liouville form.

Let us denote by $X$ the geodesic vector field, $V$ the vertical vector field and the orthogonal vector field $X_\bot := [X,V]$ (cf. \cite[Chapter 3.5]{GIP2D} for more details). The following structure equations hold
\begin{equation}\label{eq: basic commutators}
    [X,V] = X_\bot, \quad [X_\bot,V] =-X,\quad [X,X_\bot] =-KV
\end{equation}
where $K$ is the Gaussian curvature of $(M,g)$ \cite[Lemma 3.5.5]{GIP2D}. The sphere bundle $SM$ is equipped with the unique Riemannian metric $G$ such that $\{X,-X_\bot,V\}$ forms a positively oriented orthonormal frame. The metric $G$ is called the Sasaki metric, and it holds that $dV_G = d\Sigma^3$ \cite[Lemma 3.5.11]{GIP2D}. We denote by $dV^1$ the volume form of $(\partial M,g)$. This leads to the definition of the volume form
\begin{equation}
    d\Sigma^2 := dV^1 \wedge dS_x
\end{equation}
on $\partial SM$. We note that $d\Sigma^2 = dV_{\partial SM}$ where on the right hand side the volume form is induced by the Sasaki metric. 

We define the following $L^2$ inner products
\begin{equation}
    (u,w)_{SM} := \int_{SM} u\overline{w}d\Sigma^3, \quad (f,g)_{\partial SM} := \int_{\partial SM} f\overline{g}d\Sigma^2.
\end{equation}
We next recall simple integration by parts formulas. For any $u, w \in C^1(SM)$, the following formulas hold \cite[Proposition 3.5.12]{GIP2D}:
\begin{align}\label{eq:int:parts}
\begin{split}
(X u, w)_{S M} & =-(u, X w)_{S M}-(\langle v, \nu\rangle u, w)_{\partial S M} \\
\left(X_{\perp} u, w\right)_{S M} & =-\left(u, X_{\perp} w\right)_{S M}-\left(\left\langle v_{\perp}, \nu\right\rangle u, w\right)_{\partial S M} \\
(V u, w)_{S M} & =-(u, V w)_{S M}.
\end{split}
\end{align}

Finally, we recall the vertical Fourier decomposition. We define the following spaces of eigenvectors of $V$
\begin{equation}\label{eq:GK-eig}
    H_k := \{\,u \in L^2(SM)\,;\, -iVu = ku\,\}, \quad \Omega_k := \{\,u\in C^\infty(SM)\,;\, -iVu=ku\,\}
\end{equation}
for any integer $k \in \Z$. It holds that any $u \in L^2(SM)$ has a unique $L^2$-orthogonal decomposition
\begin{equation}\label{eq:orthogonal}
    u=\sum_{k=-\infty}^\infty u_k, \quad \norm{u}^2 = \sum_{k=-\infty}^\infty \norm{u_k}^2, \quad u_k \in H_k.
\end{equation}
 If $u \in C^\infty(SM)$, then $u_k \in C^\infty(SM)$ and the series converges in $C^\infty(SM)$.

We next discuss some important boundary operators following \cite[Lemma 4.5.4]{GIP2D}. We define the tangential vector field $T$ on $\partial SM$ by setting that
\begin{equation}
    T := (V\mu)X + \mu X_\bot |_{\partial SM}
\end{equation}
where $\mu(x,v) := \ip{\nu(x)}{v}_g$. The Pestov identity with boundary terms is due to Ilmavirta and Salo \cite[Lemma 8]{Ilmavirta:Salo:2016}; see also \cite[Proposition 4.5.5]{GIP2D} and the higher dimensional generalization in \cite[Lemma 8]{Ilmavirta:Paternain:2022}. For any $u \in C^2(SM)$ it holds that
\begin{equation}
    \label{eq:basic-pestov-with-bdry}
    \norm{VXu}^2_{SM}=\norm{XVu}^2_{SM} -(KVu,Vu)_{SM}+\norm{Xu}^2_{SM} + (Tu,Vu)_{\partial SM}
\end{equation}
whenever $(M,g)$ is a compact Riemannian surface with smooth boundary. We generalize \eqref{eq:basic-pestov-with-bdry} in Proposition \ref{lm:pestov:g} to the case where $X$ is replaced by the generator of a twisted geodesic flow.

\subsection{Twisted geodesic flows}
\label{sec:twisted}
We first recall the concept of $\lambda$-geodesic flows from the lectures of Merry and Paternain \cite[Chapter 7]{Marry:Paternain:2011:notes}. Let $(M,g)$ be a complete oriented Riemannian surface (with or without boundary) and $\lambda \in C^\infty(SM)$ be a smooth real valued function. We say that a curve $\gamma: [a,b] \to M$ is a \emph{$\lambda$-geodesic} if it satisfies the \emph{$\lambda$-geodesic equation}
\begin{equation}\label{eq:lambda-geodesic-eq}
D_t\dot{\gamma}=\lambda(\gamma,\dot{\gamma})i\dot{\gamma}.
\end{equation}
When $\lambda\equiv0$, then the $\lambda$-geodesics are the usual geodesics of $(M,g)$.  One may think that the function $\lambda$ twists the usual geodesics in order to model trajectories of particles when moving in the presence of external forces. When $\lambda$ is a smooth function on $M$ or $1$-form, then the twisted geodesics correspond to the magnetic and thermostatic geodesics, respectively. For other advances in the context of inverse problems for $\lambda$-geodesics, we refer to \cite{Assylbekov:Dairbekov:2018,Zhang:2023}. In particular, the class of $\lambda$-geodesics is large and can be characterized with only three natural properties of a curve family, for details see \cite[Theorem 1.4]{Assylbekov:Dairbekov:2018}.

Let $\gamma_{x,v}$ be the unique $\lambda$-geodesic  with the initial condition $(\gamma_{x,v}(t),\dot{\gamma}_{x,v}(t))=(x,v)$ and solving \eqref{eq:lambda-geodesic-eq}. As in \cite[Exercise 7.2]{Marry:Paternain:2011:notes}, one may define the \emph{$\lambda$-geodesic flow} by setting that
\begin{equation}
    \phi_t: SM \to SM,\quad \phi_t(x,v) = (\gamma_{x,v}(t),\dot{\gamma}_{x,v}(t)).
\end{equation}
The \emph{infinitesimal generator of the $\lambda$-geodesic flow} $F$ is given by
\begin{equation}
    F = X+\lambda V.
 \end{equation}
Using \eqref{eq: basic commutators}, one may derive the following commutator formulas \cite[p. 537]{Dairbekov:Paternain:2007:Entropy1}:
\begin{equation}\label{eq: twisted-commutators}
    [V,F]=-X_\bot + V(\lambda)V, \quad [V,X_\bot]=F-\lambda V, \quad [F,X_\bot] = \lambda F - (K+X_\bot(\lambda)+\lambda^2)V.
\end{equation}
We use the following notations \begin{equation}
    P := VF, \quad \widetilde{P} := FV.
\end{equation}
Notice that the formal $L^2$ adjoint of the operator $P$ is $P^* =(F+V(\lambda))V \neq \widetilde{P}$. Let us define the \emph{$\lambda$-curvature} of $(M,g,\lambda)$ as a map in $C^\infty(SM)$ by the formula 
\begin{equation}
K_{\lambda} := K+X_\bot(\lambda)+\lambda^2+F(V(\lambda)).
\end{equation}
We can now recall the following generalized Pestov identity by Dairbekov and Paternain for closed surfaces $M$ \cite[Theorem 3.3]{Dairbekov:Paternain:2007:Entropy1}:
\begin{equation}\label{eq: general Pestov}
    \norm{Pu}_{SM}^2 = \norm{\widetilde{P}u}_{SM}^2 - (K_{\lambda} Vu,Vu)_{SM}+\norm{Fu}_{SM}^2
\end{equation}
for any $u \in C^\infty(SM)$. This also holds on compact surfaces with boundary if one additionally assumes that $u|_{\partial SM}=0$. We generalize \eqref{eq:basic-pestov-with-bdry} and \eqref{eq: general Pestov} to the surfaces with smooth boundary without making the simplifying assumption that $u|_{\partial SM}=0$ (cf. Proposition \ref{lm:pestov:g}). We also remark that similar Pestov identities, but in a slightly less explicit form, were derived by Assylbekov and Dairbekov for the $\lambda$-geodesic flows on Finslerian surfaces \cite[Theorem 2.3]{Assylbekov:Dairbekov:2018}. In turn, the generalized Pestov identity with boundary terms is used to study generalized broken ray transforms.

We define the \emph{signed $\lambda$-curvature} as 
\begin{align}\label{eq:kappa:lam}
  \kappa_{\lambda}(x,v)&:=  \kappa+\left\langle v_{\perp}, \nu\right\rangle \lambda(x, v)=\kappa-\langle\lambda(x, v) i v, \nu(x)\rangle=\kappa-\langle \nu(x),\lambda(x, v) i v\rangle.
\end{align}
Additionally, we introduce a related term:
\begin{equation}\label{eq:eta:lam}
    \eta_{\lambda}(x,v):=\langle V(\lambda)(x, v) v, \nu\rangle,
\end{equation}
which will appear in the Pestov identity \eqref{eq:Pestov3.1}. We remark that $\kappa_\lambda$ and $\eta_\lambda$ depend only on the values of $\lambda$ on $\partial SM$. For further details, see Subsection \ref{sec:signed-preliminaires}.

\subsubsection{Magnetic flows}\label{sec:magnetic} We refer to the articles of Arnold \cite{arnold-61-magnetic-flows} and Anosov--Sinai \cite{Anosov-Sinai-67-magnetic} as first mathematical studies of magnetic flows. We will mainly follow the notation used by Ainsworth in the series of works \cite{Ainsworth:2013,Ainsworth:2015,Ainsworth-Assylbekov-2015-magnetic-range}, considering the integral geometry of magnetic flows. We further note the following works related to different inverse problems for magnetic flows \cite{Assylbekov-Zhou-2015-magnetic-rigidity-potential,Dairbekov-Paternain-SU-2007-magnetic-rigidity,Herros2012-magnetic-scattering-rigidity,Herros-Vargo-2011,MR2373809,MR2320165}, including the boundary, lens and scattering rigidity problems.

Let $\Omega$ be a closed $2$-form on $M$ modeling a magnetic field. The Lorentz force $Y: TM \to TM$ associated with the magnetic field $\Omega$ is the unique bundle map such that
\begin{equation}
\Omega_x(\xi,\eta)=\ip{Y_x(\xi)}{\eta}_g, \quad \forall x \in M, \ \xi,\eta \in T_xM.
\end{equation}
We say that $\gamma$ is a \emph{magnetic geodesic} if it satisfies the magnetic geodesic equation
\begin{equation}\label{eq: magnetic geodesic equation}
D_t\dot{\gamma}=Y(\dot{\gamma}).
\end{equation}
Notice now that since $M$ is orientable, there exists a unique function $\tilde{\lambda}: M \to \R$ such that $\Omega = \tilde{\lambda}dV^2$. We may define $\lambda = \tilde{\lambda} \circ \pi$. Now it holds that $\gamma$ solves \eqref{eq: magnetic geodesic equation} if and only if it solves \eqref{eq:lambda-geodesic-eq}. We may define the \emph{magnetic flow} simply as the corresponding $\lambda$-geodesic flow with the fixed energy level $1/2$ corresponding to the unit speed curves.

One may also view the magnetic flow as the Hamiltonian flow of $H(x,v) = \frac{1}{2}\abs{v}_g^2$ under the symplectic form
\begin{equation}
    \omega := \omega_0 + \pi^*\Omega
\end{equation}
where $\omega_0$ is the symplectic structure of $TM$ generated by the metric pullback of the canonical symplectic form on $T^*M$. The magnetic geodesics are known to be constant speed and different energy levels lead to different curves. We also remark that the magnetic flow is time-reversible if and only if $\Omega\equiv0$. Therefore, $\gamma_{x,v}(-t)$ is not a magnetic geodesic of $(M,g,\Omega)$. However, we have that the magnetic field with flipped sign $-\Omega$ reverses the orientation of geodesics, i.e. $\gamma_{x,-v}^{-\Omega}(t)=\gamma_{x,v}^{\Omega}(-t).$ One may check that $-\Omega$ is the dual of $\Omega$ in the sense of Section \ref{subsec:dualflows}.
\subsubsection{Thermostatic flows}\label{sec:thermostat}
The Gaussian thermostats concept was proposed by Hoover for the analysis of dynamical systems in mechanics \cite{Hoover:1986}, but it appears also earlier, for example, in the work of Smale \cite{Smale67-differentiable-dynamical-systems-thermostats}. The inverse problem for Gaussian thermostats has been more recently explored in the works of Dairbekov and Paternain \cite{Dairbekov:Paternain:2007:Entropy1,Dairbekov:Paternain:2007:II}. Other contributions have also been made by Assylbekov and Zhou \cite{Assylbekov:Zhou:2017}, Assylbekov and Dairbekov \cite{Assylbekov:Dairbekov:2018}, and Assylbekov and Rea \cite{Assylbekov:Rea:Arxiv:2021}. In addition, the dynamical and geometrical properties of Gaussian thermostats have been extensively studied, as demonstrated in the contributions by Wojtkowski \cite{Wojtkowski:2000,Wojtkowski:2000b}, Paternain \cite{Paternain:2007}, Assylbekov and Dairbekov \cite{Assylbekov:Dairbekov:2014}, and Mettler and Paternain \cite{Mettler:Paternain:2019,Mettler:Paternain:2020}. Gaussian thermostat also arises in Weyl geometry, see for instance \cite{Przytycki:Wojtkowski:2008}.

Consider a smooth vector field $E$ on $M$, representing an external field. A \emph{thermostatic geodesic} satisfies the equation
\begin{align}\label{eq:thermostat:geodesic}
D_t \dot{\gamma}=E(\gamma)-\frac{\langle E(\gamma), \dot{\gamma}\rangle}{|\dot{\gamma}|^2} \dot{\gamma}.
\end{align}
The flow $\phi_t=(\gamma(t), \dot{\gamma}(t))$ is called as thermostatic flow.
It is worth mentioning that thermostatic geodesics are time-reversible, meaning that $\phi_t(x,-v)=r\left(\phi_{-t}(x, v)\right)$, where $r: S M \rightarrow S M$ is the reversion map defined by $r(x, v)=(x,-v)$. When $E=0$, then the thermostatic geodesics are the usual geodesic. Given a $1$-form $\lambda$ defined by $\lambda(x, v):=\langle E(x), i v\rangle$, the equation \eqref{eq:thermostat:geodesic} can be rewritten as the corresponding $\lambda$-geodesic equation
\begin{align}
D_t \dot{\gamma}=\lambda(\gamma, \dot{\gamma}) i \dot{\gamma}.
\end{align}

\subsection{Dual \texorpdfstring{$\lambda$}{lambda}-geodesic flow}
\label{subsec:dualflows} 
It is well known that the usual geodesics are time-reversible. The magnetic geodesics are not time reversible unless $\Omega\equiv 0$ (cf. \cite[p. 537]{Dairbekov-Paternain-SU-2007-magnetic-rigidity}). In  \cite[p. 100]{Benedetti:Kang:2022}, it is mentioned that the magnetic geodesics corresponding to $\lambda$ and $-\lambda$ are one-to-one correspondence through time reversal. This means that a curve $t \mapsto \gamma_{x,v}(t)$ is a magnetic $\lambda$-geodesic if and only if the curve $t \mapsto \gamma_{x,v}(-t)$ is a magnetic $(-\lambda)$-geodesic. However, the thermostatic $\lambda$-geodesic flow is time-reversible, see for instance in \cite[p. 88]{Paternain:2007}). Therefore, the $\lambda$-geodesic flow in the case of $\lambda \in C^\infty(SM)$ is in general not time-reversible. 

Next, we will define the corresponding dual $\lambda$-geodesic flow. We can define the time-reversed dynamical system related to $\lambda$ by setting \begin{equation}\lambda^{-}(x,v):=-\lambda(x,-v).\end{equation} It now follows that $\gamma_{x,-v}^{-}(t)=\gamma_{x,v}(-t)$ where $\gamma_{x,-v}^{-}$ is a unique $\lambda^-$-geodesic with initial data $(x,-v)$. We call $\lambda^{-}$ as the \emph{dual} of $\lambda$. This time-reversibility property can be checked by substituting $\gamma_{x,-v}^{-}(t):=\gamma_{x,v}(-t)$ to the $\lambda^{-}$-geodesic equation and using the fact that $\gamma_{x,v}$ solves the $\lambda$-geodesic equation. In fact, 
\begin{align*}
    \nabla_{\dot\gamma^{-}}\dot\gamma^{-}|_{t=s}&=\lambda^{-}(\gamma(-s),\frac{d}{dt}(\gamma(-t))|_{t=s})i\left[\frac{d}{dt}(\gamma(-t))|_{t=s}\right]\\
    &=-\lambda(\gamma(-s),-(-\dot\gamma(-s)))i(-\dot\gamma(-s))\\
    &=\lambda(\gamma(-s),\dot\gamma(-s) )i(\dot\gamma(-s))
\end{align*}
and in local coordinates
\begin{align*}\nabla_{\dot\gamma}\dot\gamma|_{t=-s}
&=\ddot{\gamma}^l(-s)+\Gamma_{j k}^l(\gamma(-s)) \dot{\gamma}^j(-s) \dot{\gamma}^k(-s) \\
&=\nabla_{\dot\gamma^-}\dot\gamma^-|_{t=s}
\end{align*}
where we write simply $\gamma=\gamma_{x,v}$ and $\gamma^{-}=\gamma_{x,-v}^{-}$. Since $\nabla_{\dot{\gamma}}\dot{\gamma}=\lambda(\gamma,\dot{\gamma})i\dot\gamma$ holds for all times $s$ in the maximal domain of the $\lambda$-geodesic $\gamma_{x,v}$, we can conclude that $\gamma^{-}$ is a $\lambda^{-}$-geodesic. On the other hand, $\dot{\gamma}^{-}(0)=-v$ and $\gamma^{-}(0)=x$, which justifies writing $\gamma^{-}=\gamma_{x,-v}^{-}$. 

The generator of the dual $\lambda$-geodesic flow $\phi_t^{-}$ is simply given by $F^{-}:=X+\lambda^{-}V$. Additionally, it is worth noting that $(\lambda^{-})^{-}=\lambda$. We will use the dual flow to establish regularity results for the solutions of a broken transport equation in Sections \ref{sec:dual-transport-equation} and \ref{sec:regularity-of-solutions}. For the sake of completeness, we will discuss in Section \ref{subsec:dualflows-curva} the curvature and signed curvature of $\lambda$ and $\lambda^{-}$.

\subsection{Lemmas for twisted geodesic flows}

In the following proposition, we provide a generalized version of the Pestov identity for the generators of twisted geodesic flows. Similar identities are proved earlier in \cite[Theorem 3.3]{Dairbekov:Paternain:2007:Entropy1} for closed Riemannian surfaces, in \cite[Lemma 8]{Ilmavirta:Salo:2016} for surfaces with boundary under the condition $\lambda \equiv 0$, and for Finslerian surfaces in terms of Lie derivatives on the boundary \cite[Theorem 2.3]{Assylbekov:Dairbekov:2018}. A detailed proof is given in Appendix \ref{sec:proof-of-pestov}.
\begin{proposition}[Generalized Pestov identity]\label{lm:pestov:g} Let $(M,g)$ be a compact Riemannian surface with smooth boundary and $\lambda \in C^\infty(SM)$. If $u\in C^2(SM)$, then we have 
    \begin{align}\label{eq:Pestov3.1}
 \begin{split}
       \|Pu\|_{SM}^2&=  \|\widetilde{P}u \|_{SM}^2-( K_{\lambda}Vu,Vu )_{SM} + \|F u\|_{SM}^2+ \left(\nabla_{T,\lambda}u, Vu\right)_{\partial SM}
 \end{split}
\end{align}
where $\nabla_{T,\lambda}:=-\langle v_{\perp},\nu\rangle F-\langle v,\nu\rangle (V(\lambda))V+\langle v,\nu\rangle X_{\perp}.$
\end{proposition}

We say that a vector field $J$ along a $\lambda$-geodesic $\gamma$ is a $\lambda$-Jacobi field if it is a variation field of $\gamma$ through $\lambda$-geodesics, i.e. $J(t) = \partial_s \gamma_s(t)|_{s=0}$ where $\gamma_s(t)$ is a smooth one-parameter family of $\lambda$-geodesics with $\gamma = \gamma_0$ (for further details we refer to Section \ref{sec: lambda-jacobi-fields}). We will next state a useful estimate on the growth rate of $\lambda$-Jacobi fields. On compact Riemannian surfaces the norm of a Jacobi field and its covariant derivative grow at most exponentially (see e.g. \cite[Lemma 10]{Ilmavirta:Salo:2016}). Such inequalities are useful in studying the stability of geodesics and their relation to the curvature of the manifold. In the following lemma, we establish a similar result for the Jacobi equation associated with $\lambda$-geodesics. A detailed proof is given at the end of Section \ref{sec: lambda-jacobi-fields}.
\begin{lemma}\label{lm:Jacobi:est} Let $(M, g)$ be a compact Riemannian surface with or without boundary and $\lambda \in C^{\infty}(SM)$. Let $J$ be a $\lambda$-Jacobi field defined along a unit speed $\lambda$-geodesic $\gamma:[a,b]\to M$. Then the following growth estimate holds for all $t \in [a,b]$
\begin{equation}\label{eq:jac_est}
    |J(t)|^2+\left|D_t J(t)\right|^2 \leq e^{C t}\left(|J(a)|^2+\left|D_t J(a)\right|^2\right),
\end{equation}
where $C$ is a uniform constant depending only on $M,g$ and $\lambda$.
\end{lemma}

\subsection{Even and odd decomposition with respect to the reflection map}

Given a Riemannian surface $(M,g)$ with smooth boundary. We define the reflection map $\rho: \partial SM \rightarrow \partial SM$ by
\[
\rho(x, v)=\left(x, v-2\langle v, \nu(x)\rangle_g \nu(x)\right).
\] 
We denote by $\rho^*$ the pull back of $\rho.$ The even and odd parts of $u: SM \to \R$ with respect to the reflection map $\rho$ are defined by the formula 
\begin{align}\label{eq:u:even:odd:def}
u_e = \frac{1}{2}(u + u \circ \rho), \qquad u_o = \frac{1}{2}(u - u \circ \rho).
\end{align}
We will next state a simple lemma related to the reflection and rotation maps. We omit presenting a proof.
\begin{lemma}\label{lm:irhoi} Let $(M,g)$ be a Riemannian surface with smooth boundary. The reflection and the rotation maps satisfy
\begin{enumerate}[(i)]
\item  $i^{-1}=-i$; \label{l:1}
\item $\rho^{-1}=\rho$; \label{l:2} 
    \item $i\rho i=\rho$.\label{l:3}
\end{enumerate}
\end{lemma}

The boundary operators $\kappa_\lambda$ and $\nu_\lambda$, as defined in \eqref{eq:kappa:lam} and \eqref{eq:eta:lam}, satisfy the following simple identities. These formulas will later on allow us to simplify the boundary terms in \eqref{lm:pestov:g}.
\begin{lemma}\label{lm:property:rho:kappa:eta}
Let $(M,g)$ be Riemannian surface with smooth boundary. Then $\kappa_{\lambda}$ and $\eta_{\lambda}$ satisfy the following properties
\begin{enumerate}[(i)]
    \item $\kappa_{\lambda\circ \rho}=  \kappa_{\lambda}\circ \rho$\label{eq:kappa:circ:1};
    \item $\eta_{\lambda\circ \rho}=  \eta_{\lambda}\circ \rho$\label{eq:eta:circ:1};
    \item $\kappa_{\lambda_e}=\left(\kappa_\lambda\right)_e$ and $\eta_{\lambda_e}=\left(\eta_\lambda\right)_e$\label{eq:kappa:circ:2};
    \item $\rho^*\left(\kappa_{\lambda_e}\right)=\kappa_{\lambda_e}$ and $\rho^*\left(\eta_{\lambda_e}\right)=\eta_{\lambda_e}.$ \label{eq:eta:circ:2}
\end{enumerate}
\end{lemma}
\noindent
A proof of Lemma \ref{lm:property:rho:kappa:eta} is given in Appendix \ref{sec:auxiliary-proofs-appendix}.

\section{Transport equation for the broken \texorpdfstring{$\lambda$}{lambda}-geodesics}
\subsection{Broken ray transforms and the transport equation}
Let $(M,g)$ be an orientable, compact Riemannian surface with smooth boundary and $\lambda \in C^\infty(SM)$.
We assume that $\partial M$ can be decomposed to the union of two disjoint relatively open disjoint subsets $\mathcal{E}$ and $\mathcal{R}$. In particular, one could think of $M$ as $M=\hat{M} \setminus O$, where $\hat{M}$ is a larger manifold containing $M$ and $O$ being an open obstacle. In this case, $\mathcal{E}=\partial \hat{M}$ and $\mathcal{R}=\partial O$ are the outer and inner boundaries of $M$. We call $\mathcal{E}$ the $\emph{emitter}$ and $\mathcal{R} $ the $\emph{reflector}$ 
of $M.$ 
 In Section \ref{subsec:dualflows}, we denote by $\phi_t$ the usual $\lambda$-geodesic flow and by $\phi_t^{-}$ its dual flow. By abuse of notation, we now continue to write the same notation  $\phi_t$ for the broken $\lambda$-geodesic flow and $\phi_t^{-}$ to its dual flow. For any $(x,v) \in SM$, we define the \emph{forward} and \emph{dual travel times}  by \begin{equation}
    \tau_{x,v} := \inf\{\,t\geq0 \,;\,\phi_t(x,v) \in \mathcal{E}\,\}, \quad \tau_{x,v}^{-} := \inf\{\,t\geq0 \,;\,\phi_t^{-}(x,v)\in \mathcal{E}\,\}.
\end{equation}
We note that for a typical $\lambda \in C^\infty(SM)$ it actually holds that $\tau_{x,v} \neq \tau_{x,v}^{-}$ for most of $(x,v) \in SM$ since the twisted geodesic flows are not reversible. On the other hand, the maximal domain of $\gamma_{x,v}$ is $[-\tau_{x,-v}^-,\tau_{x,v}]$, and that of $\gamma_{x,-v}^{-}$ is $[-\tau_{x,v},\tau_{x,-v}^-]$.
Let $\pi:SM\to M$ be a projection map so that $\pi(x,v)=x$.
We define \begin{equation}\pi^{-1}\E:=\{(x,v): x\in \E, \ v\in S_xM\},\quad \pi^{-1}\rR:=\{(x,v): x\in \rR, \ v\in S_xM\}.\end{equation}
Let us denote by $\rho: \pi^{-1}\rR \to \pi^{-1}\rR$ the \emph{reflection map} and define by the law of reflection
\begin{equation}
   \rho(x,v) := \left(x,v-2\ip{v}{\nu(x)}_g\nu(x)\right).
\end{equation}
Note that $\rho$ is an involution in the sense that $\rho \circ \rho = \id$. Here and subsequently, $\gamma(t_0-)$ stands for the left-hand limit of $\gamma$ at some point $t_0$ and $\gamma(t_0+)$ denotes the right-hand limit of $\gamma$ at $t_0$, which are defined by $\gamma(t_0-)=\lim_{t\to t_0^-}\gamma(t)$ and $\gamma(t_0+)=\lim_{t\to t_0^+}\gamma(t)$. Similarly, $\dot{\gamma}(t_0-)=\lim_{t\to t_0^-}\dot{\gamma}(t)$ and $\dot{\gamma}(t_0+)=\lim_{t\to t_0^+}\dot{\gamma}(t)$.
\begin{definition} Let $(M,g)$ be a Riemannian surface with smooth boundary and $\lambda \in C^\infty(SM)$. Assume that $\partial M = \E \cup \rR$ where $\E$ and $\rR$ are  two relatively open disjoint subsets. 
A curve $\gamma$ on $M$ is called a \emph{broken $\lambda$-ray} if $\gamma$ is a $\lambda$-geodesic in $\inte(M)$ and reflects on $\mathcal{R}$ according to the law of reflection \begin{equation}\rho(\gamma(t_0-),\dot{\gamma}(t_0-))=(\gamma(t_0+),\dot{\gamma}(t_0+))\end{equation} whenever there is a reflection at $\gamma(t_0) \in \mathcal{R}$.
\end{definition}
The broken $\lambda$-ray transform of $f\in C^2(SM)$ is defined by
\begin{equation}\label{eq:integral_equation}
If(x,v)=\int_0^{\tau_{x,v}}f(\phi_t(x,v))dt, \ \ \ (x,v) \in \pi^{-1}\E,
\end{equation}
where $\phi_t(x,v) = (\gamma_{x,v}(t), \dot\gamma_{x,v}(t))$ is the broken $\lambda$-geodesic flow.
We next move towards studying the injectivity of the broken $\lambda$-ray transform, that is, is it possible to determine an unknown $f$ from the knowledge of its integrals \eqref{eq:integral_equation} over maximal broken $\lambda$-rays? To proceed, we first reduce the integral equation \eqref{eq:integral_equation} to a certain partial differential equation. Given any $f\in C^2(SM)$, define
\begin{equation}\label{eq:u:def}
u(x,v) :=\int_0^{\tau_{x,v}}f(\phi_t(x,v))dt, \ \ \ (x,v) \in SM.
\end{equation}
Notice that the exit time $\tau_{x,v}$ is smooth near  $(x,v) \in SM$ whenever the broken ray $\gamma_{x,v}$ reflects and exits transversely. 
A simple application of the fundamental theorem of calculus together with the regularity properties of exit time, we deduce from \eqref{eq:u:def} that $u$ satisfies the transport equation 
\begin{align}\label{eq:trans:1}
    \begin{split}
         \begin{cases}
(X+\lambda V)u=-f, & \text{ in } \operatorname{int}(SM), \\
u|_{\pi^{-1}\E}=If, & u|_{\pi^{-1}\rR}= u\circ \rho |_{\pi^{-1}\rR}.
\end{cases}
    \end{split}
\end{align}
We need to make some assumption on the geometry of $M$ and its reflecting boundary parts to make the injectivity problem more approachable.
\begin{remark}
   If $\lambda\circ \rho=\lambda$ on $\rR$, then we have 
   \begin{align*}
\left(\kappa_\lambda\right)_o+\left(\eta_\lambda\right)_o=0,  \end{align*}
where $\left(\kappa_\lambda\right)_o$ and $\left(\eta_\lambda\right)_o$ are odd parts of the functions $\kappa_{\lambda}$ and $\eta_{\lambda}$ respectively. This assumption holds, for example, when $\lambda$ is independent of the direction $v$ on $\pi^{-1}\rR$.
\end{remark}
\begin{definition}\label{def:lambda:adm:r}
Let $(M, g)$ be a Riemannian surface with smooth boundary and $\lambda \in C^{\infty}(SM)$. We say that a relatively open subset $\rR \subset \partial M$ of the boundary has an \emph{admissible $\lambda$-curvature} if the following inequality holds:
\begin{equation}
    (\kappa_{\lambda})_e(x,v) + (\eta_{\lambda})_e(x,v) \le 0 \quad \text{ for all } (x,v) \in \pi^{-1}\rR.
\end{equation}
\end{definition}
\noindent If $V(\lambda)$ vanishes on $\pi^{-1}\rR$, i.e. $\lambda$ is only a function of the base point on $\rR$, and $\rR$ is strictly $\lambda$-concave at any $x \in \rR$, then $\rR$ has an admissible $\lambda$-curvature.

From Corollary \ref{cor:dual-curvatures} and Remark \ref{rmk:eta-dual-law}, we have 
\begin{align*}
    (\kappa_{\lambda^{-}})_e(x, v)&+(\eta_{\lambda^{-}})_e(x, v)=\kappa_{\lambda_e^{-}}(x, v)+\eta_{\lambda_e^{-}}(x, v)\\&=\kappa_{\lambda_e}(x, -v)+\eta_{\lambda_e}(x, -v)=(\kappa_{\lambda})_e(x, -v)+(\eta_{\lambda})_e(x, -v),
\end{align*}
i.e., an obstacle $\rR$ has admissible $\lambda$-curvature if and only if $\rR$ has admissible $\lambda^{-}$-curvature.  
If $V(\lambda_e)|_{\rR}=0$, then the condition that the obstacle $\rR$ has admissible $\lambda$-curvature is equivalent to the strict $\lambda$-concavity of $\rR$.

In \cite[Theorem 1]{Ilmavirta:Salo:2016}, it was proved that one can recover an unknown function $f$ from the knowledge of its geodesic broken ray transform $If$. They assumed that the surface is nontrapping, having nonpositive Gaussian curvature, the reflecting part is strictly concave, and the broken rays allow at most one reflection with $|\langle\dot{\gamma},\nu\rangle|<a$. See also  \cite[Definition 1]{Ilmavirta:Paternain:2022} for similar assumptions used to study the broken ray transforms in three and higher dimensions. We now define a similar class of admissible Riemannian surfaces with broken $\lambda$-geodesic flows.
\begin{definition}\label{def:admis:lam:brok} 
Let $(M,g)$ be a compact Riemannian surface with smooth boundary and $\lambda \in C^\infty(SM)$. Assume that $\partial M = \E \cup \rR$ where $\E$ and $\rR$ are two relatively open disjoint subsets.
We call a dynamical system $(M,g,\lambda,\mathcal{E})$ \emph{admissible} if
\begin{enumerate}[(1)]
\item the emitter $\mathcal{E}$ is strictly $\lambda$-convex in the sense of  Definition \ref{def:lam_convex};
\item the obstacle $\mathcal{R}$ has admissible $\lambda$-curvature in the sense of Definition \ref{def:lambda:adm:r};
\item the $\lambda$-curvature $K_{\lambda}$ of $(M,g,\lambda)$ is nonpositive;
\item \label{eq:nontrapping}the broken $\lambda$-geodesic flow is nontrapping: there exists $L > 0$ such that $\tau_{x,v}, \tau_{x,v}^{-} < L$ for any $(x,v) \in SM$;
\item there exists $a > 0$ such that every broken $\lambda$-ray $\gamma$ has at most one reflection at $\mathcal{R}$ with $\abs{\ip{\nu}{\dot{\gamma}}}<a$.
\end{enumerate}
\end{definition}

\subsection{Dual transport equation}\label{sec:dual-transport-equation}
Let us define the \emph{dual transport equation} of \eqref{eq:trans:1} as
\begin{align}\label{eq: dual transport eq}
    \begin{split}
         \begin{cases}
 (X+\lambda^{-}V)u=-\tilde{f}, & \text{ in } \operatorname{int}(SM), \\
u|_{\pi^{-1}\E}=I^-\tilde{f}, & u|_{\pi^{-1}\rR}= u\circ \rho |_{\pi^{-1}\rR},
\end{cases}
    \end{split}
\end{align}
where $\tilde{f}(x,v) :=f(x,-v)$ and $I^-$ denote the broken ray transform related to the $\lambda^-$-geodesic flow. For nontrapping broken $\lambda$-geodesic flows, we define the \emph{scattering relation} $\alpha: \partial SM \rightarrow \partial SM$ by
  \begin{align*}
      \alpha(x,v):=\phi_{\tau_{x,v}}(x,v).
  \end{align*}
  
\begin{lemma}\label{dual_tr_lem}
    Let $(M,g)$ be a compact nontrapping Riemannian surface with smooth boundary and $\lambda \in C^\infty(SM)$. If $f \in C^2(SM)$ and $If=0$, then \begin{equation} \label{eq: transport-dual-relation} u(x,v) = -u^{-}(x,-v)\end{equation}
where $u$ is the unique solution of the transport equation \eqref{eq:trans:1} and $u^{-}$ is the unique solution of the dual transport equation \eqref{eq: dual transport eq}.
\end{lemma}
\begin{proof}
 Let us consider $(x, v) \in SM$.
Now the union of curves $\gamma_{x, v}$ and $\gamma^{-}_{x, -v}$ form a maximal broken $\lambda$-geodesic (cf. Section \ref{subsec:dualflows}) with its endpoints on $\partial M$. By assumption, $If=0$, we can deduce that
 \begin{align}\label{eq: If=0 property}
     0&=\int_0^{\tau_{x,v}} f(\gamma_{x,v}(t),\dot\gamma_{x,v}(t) )dt+\int_0^{\tau_{x,-v}^-} f(\gamma^-_{x,-v}(t),-\dot\gamma^-_{x,-v}(t) )dt.
 \end{align}
By the definition $u^{-}$ is the unique solution of \eqref{eq: dual transport eq}. Note that \[If(x,v)=I^-\tilde{f}(r\circ \alpha)(x,v)\]
where $\alpha$ is the scattering relation and $r$ is the reversion map. This implies that $u^-|_{\pi^{-1}\E}=I^{-} \tilde{f}=0$.
Since $X+\lambda^{-}V$ is the generator of $\lambda^{-}$-geodesic flow  $\phi^{-}$, we get that
\begin{equation}\label{eq: u-minus from def}
    u^{-}(x,v) = \int_0^{\tau_{x,v}^{-}} \tilde{f}(\gamma_{x,v}^{-}(t),\dot{\gamma}_{x,v}^{-}(t))dt = \int_0^{\tau_{x,v}^{-}} f(\gamma_{x,v}^{-}(t),-\dot{\gamma}_{x,v}^{-}(t))dt.
\end{equation}
The formulas \eqref{eq: If=0 property} and \eqref{eq: u-minus from def} imply that 
\begin{equation}
    u(x,v) + u^{-}(x,-v) = 0,
\end{equation}
which completes the proof.\qedhere
\end{proof}

\begin{remark} \label{rmk:dual_tr_lem}
 Lemma \ref{dual_tr_lem} also clearly holds in the setting of admissible dynamical systems $(M,g,\lambda,\mathcal{E})$.
 \end{remark}
\subsection{Jacobi fields for broken \texorpdfstring{$\lambda$}{lambda}-geodesics}
In this subsection, we give a brief exposition of Jacobi fields along broken $\lambda$-geodesic flows. In Lemma \ref{lm:converse:jacobi}, we show that Jacobi fields along a $\lambda$-geodesic flow can be realized as the variation field of a unit speed $\lambda$-geodesic variation of $\gamma$. Here, we will generalize similar properties of Jacobi fields to the case of broken $\lambda$-geodesic flows. The crucial point is to understand how the Jacobi fields behave at reflection points. The Jacobi fields along broken rays have been studied in the case of $\lambda =0$ in \cite[Section 5]{Ilmavirta:Salo:2016} and we follow some of the techniques from this article.
Let $x_0\in\partial M$ and $\nu$ be the inward unit normal to it. We define a map $\Phi_{\zeta}: T_{x_{0}} M \rightarrow T_{x_{0}} M$ by setting
\[
\Phi_{\zeta} \xi=2\left(\left\langle\nabla_{\varphi_{\zeta} \xi} \nu, \zeta\right\rangle \nu+\langle\nu, \zeta\rangle \nabla_{\varphi_{\zeta} \xi} \nu\right), 
\]
for any vector $\zeta \in T_{x_{0}} M$ that is not orthogonal to $\nu$, where the map   $\varphi_{\zeta}: T_{x_{0}} M \rightarrow T_{x_{0}} M$ 
is defined by
\[\varphi_{\zeta} \xi=\xi-\frac{\langle\xi, \nu\rangle}{\langle\zeta, \nu\rangle} \zeta.
\]
Since $\varphi_{\zeta} \xi \perp \nu$, the derivative $\nabla_{\varphi_{\zeta} \xi} \nu$ is well defined. To analyze Jacobi fields at reflection points, we first make an assumption on $\lambda$ on the reflected part of the boundary. In particular, we require a condition on $\lambda$ such that
\begin{equation}
\label{eq:beta-reflection-condition}
\lambda \circ \rho = \beta \lambda \quad \text{on} \quad \rR,
\end{equation}
where $\beta \in L^{\infty}\left(\pi^{-1} \rR\right)$. We refer to this condition as \textit{the $\beta$-reflection condition on $\lambda$}. Taking $(x,v)\to \rho(x,v)$, we have
\[
\lambda\circ \rho (\rho(x,v)) =\beta (\rho(x,v) ) \lambda(\rho(x,v))
      \]
      and
      \[
       \lambda (x,v)=\beta (\rho(x,v) ) \lambda(\rho(x,v)),
      \]
   which gives the condition on $\beta$ that $1=\beta (x,v)\beta (\rho(x,v) )$. 
For example, if we consider $Z=\{(x,v):\lambda(x,v)=0\}=\{(x,v):\lambda\circ \rho(x,v)=0\}$, then we have a $\beta$ given as follows
   \begin{align*}
       \beta(x,v)= \begin{cases}
\frac{\lambda\circ \rho(x,v)}{\lambda(x,v)} & \text{ in }\pi^{-1}\rR \setminus  Z \\
1 & \text{ in }  Z. 
\end{cases}
   \end{align*}

   However, at the end of this section, we are able to establish suitable growth estimates for $\lambda$-Jacobi fields without this additional $\beta$-reflection condition. The benefit of the $\beta$-reflection condition is that it allows one to write down a clean reflection condition for the broken Jacobi fields.
\begin{definition}
 Let $(M,g)$ be a compact Riemannian surface with smooth boundary and $\lambda \in C^\infty(SM)$. Assume that $\partial M=\mathcal{E} \cup \mathcal{R}$ where $\mathcal{E}$ and $\mathcal{R}$ are  two relatively open disjoint subsets and $\lambda \circ \rho|_{\pi^{-1}\rR}=\beta \lambda|_{\pi^{-1}\rR}$ for some $\beta\in L^{\infty}(\pi^{-1}\rR)$. Let $\gamma$ be a broken $\lambda$-ray without tangential reflections. Then a vector field $J$ along $\gamma$ is a \emph{Jacobi field along $\gamma$} if 
\begin{enumerate}[(i)]
\item $J$ is a $\lambda$-Jacobi field along the segments of $\lambda$-geodesic $\gamma$ in $\inte(SM)$ in the usual sense (cf. Section \ref{sec: lambda-jacobi-fields});
\item if $\gamma$ has a reflection at $\gamma\left(t_{0}\right) \in \rR$, then the left and right limits of $J$ at $t_{0}$ are related via
\begin{equation}\label{eq:JS:9}
    \begin{split}
        &J\left(t_{0}+\right)=\rho J\left(t_{0}-\right), \ \text{and}
        \\
& D_{t} J\left(t_{0}+\right)=\rho D_{t} J(t_0-)-\Phi_{\dot{\gamma}(t_0-)}  J(t_0-)\\
&\qquad\qquad\qquad -\left(\beta\left(\gamma\left(t_0-\right), \dot{\gamma}\left(t_0-\right)\right)+1\right)\frac{\left\langle J\left(t_0-\right), \nu\right\rangle}{\left\langle I\left(t_0-\right), \nu\right\rangle} \rho  D_t I(t_0-),
    \end{split}
\end{equation}
where $I(t)=\dot{\gamma}(t)$.
\end{enumerate}
\end{definition}
It is clear that if $(\lambda \circ \rho)=\beta \lambda$ on $\pi^{-1} \mathcal{R}$, then we have  $(\lambda^- \circ \rho)=\frac{1}{\beta} \lambda^-$ on $\pi^{-1} \mathcal{R}$ and hence the identity \eqref{eq:JS:9} 
 is equivalent to 
\[
\begin{split}
   & J\left(t_{0}-\right)=\rho J\left(t_{0}+\right),  \ \text{and}
 \\
 & D_{t} J\left(t_{0}-\right)=\rho D_{t} J(t_0+)-\Phi_{\dot{\gamma}(t_0+)}  J(t_0+) \\
 &\qquad\qquad\qquad  -\left(\frac{1}{\beta\left(\gamma\left(t_0+\right), \dot{\gamma}\left(t_0+\right)\right)}+1\right)\frac{\left\langle J\left(t_0+\right), \nu\right\rangle}{\left\langle I\left(t_0+\right), \nu\right\rangle} \rho  D_t I(t_0+).
\end{split}
\]
\begin{remark}In the case of usual geodesics it holds that $D_tI=D_t\dot{\gamma}=0$.
\end{remark}

In \cite[Lemma 12]{Ilmavirta:Salo:2016}, it has been pointed out that if none of the broken geodesic rays $\gamma_{s}$ have tangential reflections, then $J(t)=\left.\partial_s\gamma_{s}(t)\right|_{s=0}$ is a Jacobi field along $\gamma_0$, where $\gamma_s$ are the variations of $\gamma_0$. Conversely, any Jacobi field can be understood as a variation of the broken geodesic $\gamma_0.$ We can now state an analogue of \cite[Lemma 12]{Ilmavirta:Salo:2016} in the case of broken $\lambda$-rays, where $\lambda\in C^{\infty}(SM)$. 
\begin{lemma}\label{lm:jac:beta:int:1}
Let $(M,g)$ be a compact Riemannian surface with smooth boundary and $\lambda \in C^\infty(SM)$. Assume that $\partial M=\mathcal{E} \cup \mathcal{R}$ where $\mathcal{E}$ and $\mathcal{R}$ are  two relatively open disjoint subsets. Let $\gamma_{s}:[0, L] \rightarrow M$ be the broken $\lambda$-ray starting at $\left(x_{s}, v_{s}\right)$ where the parametrization $(-\varepsilon, \varepsilon) \ni s \mapsto (x_s, v_s) \in \operatorname{int} SM$ is a smooth map. If the broken $\lambda$-rays $\gamma_s$ do not have tangential reflections and $  (\lambda \circ \rho)=\beta\lambda$ on $\pi^{-1}\rR$ where $\beta,\frac{1}{\beta}\in L^{\infty}(\pi^{-1}\rR)$, then 
\[
J(t)=\left.\partial_s\gamma_{s}(t)\right|_{s=0}
\]
is a Jacobi field along the broken $\lambda$-ray $\gamma_{0}$. 
\end{lemma}
\begin{proof}
   By Lemma \ref{lm:lambda:jac:var}, it follows that $J$ satisfies the $\lambda$-Jacobi equation on each $\lambda$-geodesic segment between the reflection points. Therefore, it is sufficient to prove that $J$ satisfies \eqref{eq:JS:9} at the reflection points. 
Let $\gamma := \gamma_0$ be a broken $\lambda$-ray that does not have a tangential reflection at the reflection point $t = t_0$. After possibly shrinking the domain of definition, we may assume  that $t = t_0$ is the only reflection point of $\gamma_0$ and each of the broken rays $\gamma_s$ have at most one reflection.

We begin by proving the lemma for a special family of curves corresponding to the tangential Jacobi fields. Let us consider a family of broken $\lambda$-geodesics $\gamma_s(t)=\gamma(t+s)$ with the starting point and velocity given by $(\gamma_s(0),\dot\gamma_s(0))=(x_s,v_s)$. We denote $I$ by the variation field corresponding to $\gamma_s$. Notice that
\[I(t)=\left.\partial_s\gamma_s(t)\right|_{s=0}=\dot\gamma(t)\] is a $\lambda$-Jacobi field except at the reflection point, see Lemma \ref{lm:lambda:jac:var}. 

We now analyze the behavior of $I$ at reflection point $t_0$. By the definition of the broken $\lambda$-ray, we have $I\left(t_{0}+\right)=\rho I\left(t_{0}-\right)$. Since $\gamma_s$ satisfies the $\lambda$-geodesic equation, we have $D_{t} I(t)=\lambda(\gamma(t),\dot\gamma(t))i\dot\gamma(t)$ outside the reflection point, and this leads to
\begin{align*}
    D_{t} I\left(t_{0}+\right)&=\lambda(\gamma(t_{0}+),\dot\gamma(t_{0}+))i\dot{\gamma}\left(t_0+\right)\\
    &=\lambda\circ \rho(\gamma(t_{0}-),\dot\gamma(t_{0}-))i\rho\dot{\gamma}\left(t_0-\right)\\
    &=\beta(\gamma(t_{0}-),\dot\gamma(t_{0}-) )\lambda( \gamma(t_{0}-),\dot\gamma(t_{0}-))i\rho\dot{\gamma}\left(t_0-\right).
\end{align*}
Applying Lemma \ref{lm:irhoi} to the above identity, we obtain
\begin{align*}
     D_{t} I\left(t_{0}+\right)&=-\beta(\gamma(t_{0}-),\dot\gamma(t_{0}-) )\lambda( \gamma(t_{0}-),\dot\gamma(t_{0}-))\rho i \dot{\gamma}\left(t_0-\right)\\&=-\beta(\gamma(t_{0}-),\dot\gamma(t_{0}-) )\rho D_tI(t_0-). 
\end{align*}
Hence the vector field $I$ satisfies \eqref{eq:JS:9}. 

Now it remains to prove the lemma in the case of general variations of broken $\lambda$-rays. Let $J$ denote the vector field along $\gamma$, as given in the statement of the lemma. In view of the proof of \cite[Lemma 12]{Ilmavirta:Salo:2016}, we define another vector field $\tilde{J}$ along $\gamma$ by setting
\[
\tilde{J}(t) = J(t) - \frac{\left\langle J\left(t_0-\right), \nu\right\rangle}{\left\langle I\left(t_0-\right), \nu\right\rangle} I(t). 
\] 
Since $\gamma(t)$ does not have a tangential reflection, we see that $\left\langle I\left(t_0-\right), \nu\right\rangle= \left\langle \dot \gamma\left(t_0-\right), \nu\right\rangle\neq 0$.

Note that $I(t)$ and $J(t)$ satisfy the $\lambda$-Jacobi equation except at the point of reflection, see Lemma \ref{lm:lambda:jac:var}. Therefore, by the linearity of the $\lambda$-Jacobi equation, the vector field $\tilde J$ must satisfy the $\lambda$-Jacobi equation except at the reflection points. Similar to Lemma \ref{lm:converse:jacobi}, there exists a corresponding family of broken $\lambda$-geodesic variations associated with $\tilde J$, say $\tilde{J}(t) = \partial_s\gamma_{s}(t)|_{s=0} = \partial_s\gamma_{x_s,v_s}(t)|_{s=0}$. One can make a change of order $s^2$ to $(x_s,v_s)$ without changing $\tilde J$. Since $\tilde{J}(t_0-) \,\bot\, \nu$ and $\gamma_0$ arrives to $\rR$ transversely at $t_0$, we can introduce a second order change to $(x_s,v_s)$ such that $\gamma_s(t_0) \in \rR$ for all $s \in (-\epsilon',\epsilon')$ after choosing a sufficiently small $\epsilon' \in (0,\epsilon)$ (cf. \cite[p. 396]{Ilmavirta:Salo:2016}). This variation after the second order change is explicitly given by $s \mapsto \phi_{\tau_s}(x_s,v_s)$ where $\tau_s$ is the unique element in $$\{t \in [t_0-\delta,t_0+\delta]\,;\,\phi_{t_0+t}(x_s,v_s) \in \pi^{-1}\rR\,\}$$
for a sufficiently small $\delta>0$ depending upon the choice of $\epsilon'$.
One can check that $\tau_0=0$ and $\partial_s \tau_s|_{s=0} =0$ using that $\gamma_0(t_0) \in \rR$ and $\tilde{J}(t_0-) \,\bot\, \nu$. Since the family of curves $s \mapsto \gamma_s(t_0-\delta')$, $\delta' >0$ sufficiently small, arrives in the limit $\delta' \to 0$ tangentially to $\rR$ at $s=0$, it follows that $\partial_s \tau_s|_{s=0} =0$.

We may assume without loss of generality that all $\gamma_s$ have their unique reflection at $t_0$. This implies that $\gamma_s(t_0)=\gamma_{x_s,v_s}(t_0)$ is smooth at $s=0$ and
\begin{align}\label{eq:jac:int:2}
     \tilde{J}\left(t_0+\right)=\tilde{J}\left(t_0-\right)= \rho  \tilde{J}\left(t_0-\right).
\end{align}
By \eqref{eq:symmetry:thm}, we have 
\begin{equation}\label{eq:p1:1}
    D_{t}\tilde J(t_0-)=D_{t}\partial_s \gamma_s(t_0)\big|_{s=0,t=t_0-}=\left.D_{s} \dot{\gamma}_{s}(t_0-)\right|_{s=0}
\end{equation}
and 
\begin{equation}\label{eq:p1:2}
    D_{t}\tilde J(t_0+)=D_{t}\partial_s \gamma_s(t)\big|_{s=0,t=t_0+}=\left.D_{s}  \dot{\gamma}_{s}(t_0+)\right|_{s=0}=\left.D_{s} \rho_{s} \dot{\gamma}_{s}(t_0-)\right|_{s=0}.
\end{equation}
Let us denote by $y_{s}=\gamma_{s}(t_0)$, $u_{s}=\dot{\gamma}_{s}(t_0-)$ and $\nu(y_s )=\nu_s$. 
Thus we have
\begin{align}\label{eq:q1:1}
\notag D_{s}\left(\rho_{s} u_{s}\right)=& D_{s}\left(u_{s}-2\left\langle u_{s}, \nu_{s}\right\rangle \nu_{s}\right) \\
\notag =& [D_{s} u_{s}-2\left\langle D_{s} u_{s}, \nu_{s}\right\rangle \nu_{s}] -2\left(\left\langle u_{s}, D_{s} \nu_{s}\right\rangle \nu_{s}+\left\langle u_{s}, \nu_{s}\right\rangle D_{s} \nu_{s}\right) \\
=& [\rho_{s}\left(D_{s} u_{s}\right)]-2\left(\left\langle u_{s}, \nabla_{\partial_{s} y_{s}} \nu_{s}\right\rangle \nu_{s}+\left\langle u_{s}, \nu_{s}\right\rangle \nabla_{\partial_{s} y_{s}} \nu_{s}\right).
\end{align}
Using the fact that $\tilde J(t_0-)\perp \nu$, we obtain
\begin{align*}
    \varphi_{\dot{\gamma}\left(t_0-\right)} \tilde J\left(t_0-\right) 
    =\tilde J\left(t_0-\right)-\frac{\left\langle \tilde J\left(t_0-\right), \nu\right\rangle}{\left\langle\dot{\gamma}\left(t_0-\right), \nu\right\rangle} \dot{\gamma}\left(t_0-\right) 
    =\tilde J\left(t_0-\right).
\end{align*}
We also have
\begin{equation}\label{eq:q1:2}
    \begin{split}
     & 2\left(\left\langle u_{s}, \nabla_{\partial_{s} y_{s}} \nu_{s}\right\rangle \nu_{s}+\left\langle u_{s}, \nu_{s}\right\rangle \nabla_{\partial_{s} y_{s}} \nu_{s}\right)|_{s=0} \\
 &\qquad\qquad =2\left(\left\langle \dot{\gamma}\left(t_0-\right), \nabla_{\tilde J\left(t_0-\right)} \nu\right\rangle \nu+\left\langle \dot{\gamma}\left(t_0-\right), \nu\right\rangle \nabla_{\tilde J\left(t_0-\right)} \nu\right)\\
    &\qquad\qquad =\Phi_{\dot{\gamma}(t_0-)} \tilde J(t_0-).   
    \end{split}
\end{equation}
Taking $s=0$ in \eqref{eq:q1:1}, and combining \eqref{eq:p1:1} with \eqref{eq:p1:2} and \eqref{eq:q1:2}, we get
\begin{equation}\label{eq:int:a12}
    D_{t} \tilde J(t_0+)=\rho D_{t}\tilde J(t_0-)-\Phi_{\dot{\gamma}(t_0-)} \tilde J(t_0-)
\end{equation}
Finally, we have
\begin{align*}
    J(t_0+)&=\tilde{J}(t_0+)+\frac{\left\langle J\left(t_0-\right), \nu\right\rangle}{\left\langle I\left(t_0-\right), \nu\right\rangle} I(t_0+)\\
    &=\rho\tilde{J}(t_0-)+\frac{\left\langle J\left(t_0-\right), \nu\right\rangle}{\left\langle I\left(t_0-\right), \nu\right\rangle} \rho  I(t_0-)
    =\rho J(t_0-).
\end{align*}
Since \begin{align*}
    \varphi_{\dot{\gamma}\left(t_0-\right)} \frac{\left\langle J\left(t_0-\right), \nu\right\rangle}{\left\langle \dot\gamma\left(t_0-\right), \nu\right\rangle} \dot\gamma(t_0-)
    &=\frac{\left\langle J\left(t_0-\right), \nu\right\rangle}{\left\langle \dot\gamma\left(t_0-\right), \nu\right\rangle} \dot\gamma(t_0-)-\frac{\left\langle J\left(t_0-\right), \nu\right\rangle}{\left\langle \dot\gamma\left(t_0-\right), \nu\right\rangle}\frac{\langle\dot\gamma(t_0-),\nu  \rangle}{\langle\dot\gamma(t_0-),\nu  \rangle}\dot\gamma(t_0-) \\
    &=0,
\end{align*}
and $\nabla_{ \varphi_\zeta \zeta }=0$, it turns out that
\begin{align}\label{eq:int:aa1}
   \Phi_{\dot{\gamma}\left(t_0-\right)} \frac{\left\langle J\left(t_0-\right), \nu\right\rangle}{\left\langle \dot\gamma\left(t_0-\right), \nu\right\rangle} \dot\gamma(t_0-)= \frac{\left\langle J\left(t_0-\right), \nu\right\rangle}{\left\langle \dot\gamma\left(t_0-\right), \nu\right\rangle}\Phi_{\dot{\gamma}\left(t_0-\right)} \dot\gamma(t_0-)=0. 
\end{align}
Therefore
\begin{align*}
  D_tJ(t_0+)&= D_t \tilde{J}(t_0+)+\frac{\left\langle J\left(t_0-\right), \nu\right\rangle}{\left\langle I\left(t_0-\right), \nu\right\rangle}  D_t I(t_0+)  \\
  &=\rho D_{t}\tilde J(t_0-)-\Phi_{\dot{\gamma}(t_0-)} \tilde J(t_0-)-\frac{\left\langle J\left(t_0-\right), \nu\right\rangle}{\left\langle I\left(t_0-\right), \nu\right\rangle} \beta\left(\gamma\left(t_0-\right), \dot{\gamma}\left(t_0-\right)\right) \rho D_t I(t_0-)\\
  &=\rho D_{t} J(t_0-)-\Phi_{\dot{\gamma}(t_0-)} \tilde J(t_0-) \\
  &\qquad -\left(\beta\left(\gamma\left(t_0-\right), \dot{\gamma}\left(t_0-\right)\right)+1\right)\frac{\left\langle J\left(t_0-\right), \nu\right\rangle}{\left\langle I\left(t_0-\right), \nu\right\rangle} \rho  D_t I(t_0-).
\end{align*}
From \eqref{eq:int:aa1}, the above identity becomes
\begin{align*}
   D_tJ(t_0+)  &=\rho D_{t} J(t_0-)-\Phi_{\dot{\gamma}(t_0-)}  J(t_0-) \\
   &\qquad -\left(\beta\left(\gamma\left(t_0-\right), \dot{\gamma}\left(t_0-\right)\right)+1\right)\frac{\left\langle J\left(t_0-\right), \nu\right\rangle}{\left\langle I\left(t_0-\right), \nu\right\rangle} \rho  D_t I(t_0-),
\end{align*}
which is our desired conclusion. 
\end{proof}

Recall that
\begin{align*}
    \Phi_\zeta \xi&=2\left(\left\langle\nabla_{\varphi_\zeta \xi} \nu, \zeta\right\rangle \nu+\langle\nu, \zeta\rangle \nabla_{\varphi_\zeta \xi} \nu\right),\\
    \varphi_\zeta \xi&=\xi-\frac{\langle\xi, \nu\rangle}{\langle\zeta, \nu\rangle} \zeta.
\end{align*}
Since $\varphi_\zeta \xi \perp \nu$, it follows that $\nabla_{\varphi_\zeta \xi} \nu$ is well defined. Moreover, we observe that $\nabla_{\varphi_\zeta \xi} \nu\perp \nu $. 

We now simplify the map $\Phi_{\dot{\gamma}\left(t_0-\right)}  J\left(t_0-\right)$. To do so, we first compute 
\begin{align*}
    \varphi_{\dot{\gamma}\left(t_0-\right)} J\left(t_0-\right)= J\left(t_0-\right)-\frac{\langle J\left(t_0-\right),\nu\rangle}{\langle \dot{\gamma}\left(t_0-\right),\nu\rangle }\dot{\gamma}\left(t_0-\right).
\end{align*}
By properties of covariant derivative along curves (cf. \cite[Theorem 4.24]{Lee:riem:2nd}), we have
\begin{align*}
     \nabla_{\varphi_{\dot{\gamma}\left(t_0-\right)} J\left(t_0-\right)} \nu&=\nabla_{ J\left(t_0-\right)-\frac{\langle  J\left(t_0-\right),\nu\rangle}{\langle \dot{\gamma}\left(t_0-\right),\nu\rangle }\dot{\gamma}\left(t_0-\right)}\nu\\
     &=\nabla_{ J\left(t_0-\right)}\nu-\frac{\langle  J\left(t_0-\right),\nu\rangle}{\langle \dot{\gamma}\left(t_0-\right),\nu\rangle }\nabla_{\dot{\gamma}\left(t_0-\right)}\nu.
\end{align*}
Since \begin{align*}
\left\langle\nabla_{\varphi_{\dot{\gamma}\left(t_0-\right)}  J\left(t_0-\right)} \nu, \dot{\gamma}\left(t_0-\right)\right\rangle &= \langle \nabla_{ J\left(t_0-\right)}\nu, \dot{\gamma}\left(t_0-\right) \rangle-\frac{\langle J\left(t_0-\right),\nu\rangle}{\langle \dot{\gamma}\left(t_0-\right),\nu\rangle }\langle\nabla_{\dot{\gamma}\left(t_0-\right)}\nu , \dot{\gamma}\left(t_0-\right) \rangle,
\end{align*}
we see that
\begin{align*}
 \Phi_{\dot{\gamma}\left(t_0-\right)} J\left(t_0-\right)&=2   \langle \nabla_{ J\left(t_0-\right)}\nu, \dot{\gamma}\left(t_0-\right) \rangle \nu -2\frac{\langle J\left(t_0-\right),\nu\rangle}{\langle \dot{\gamma}\left(t_0-\right),\nu\rangle }\langle\nabla_{\dot{\gamma}\left(t_0-\right)}\nu , \dot{\gamma}\left(t_0-\right) \rangle \nu\\
 &\qquad +2 \langle \nu, \dot{\gamma}\left(t_0-\right) \rangle \nabla_{ J\left(t_0-\right)}\nu-2  \langle \nu, \dot{\gamma}\left(t_0-\right) \rangle \frac{\left\langle J\left(t_0-\right), \nu\right\rangle}{\left\langle\dot{\gamma}\left(t_0-\right), \nu\right\rangle} \nabla_{\dot{\gamma}\left(t_0-\right)} \nu\\
 &=2   \langle \nabla_{ J\left(t_0-\right)}\nu, \dot{\gamma}\left(t_0-\right) \rangle\nu -2\frac{\langle  J\left(t_0-\right),\nu\rangle}{\langle \dot{\gamma}\left(t_0-\right),\nu\rangle }\langle\nabla_{\dot{\gamma}\left(t_0-\right)}\nu , \dot{\gamma}\left(t_0-\right) \rangle \nu\\
 &\qquad +2 \langle \nu, \dot{\gamma}\left(t_0-\right) \rangle \nabla_{J\left(t_0-\right)}\nu-2\left\langle J\left(t_0-\right), \nu\right\rangle  \nabla_{\dot{\gamma}\left(t_0-\right)} \nu.
\end{align*}
Note that $\Phi_{\dot{\gamma}(t_0-)}$ is a linear map. Due to the compactness and smoothness of $\rR$, similar to \cite[Remark 13]{Ilmavirta:Salo:2016}, we establish that \begin{equation}\label{eq:Phi_A_relation}
    \Phi_{\dot{\gamma}\left(t_0-\right)} J\left(t_0-\right)=\left\langle\dot{\gamma}\left(t_0-\right), \nu\right\rangle^{-1} A J\left(t_0-\right)
\end{equation}
where the field of linear maps $A$ is uniformly bounded over the set $\pi^{-1}\rR$. This follows using compactness and smoothness of $\rR$ since $A$ at $(x,v) \in \pi^{-1}\rR$ depends only $\nu$, its first derivatives and the choice of a unit vector $(x,v)=(\gamma(t_0-),\dot{\gamma}(t_0-))$. This implies that the Jacobi field tends to infinity as the reflection becomes more tangential. In particular, we have obtained the following auxiliary result.

\begin{remark} The reflection condition on a Jacobi field along a broken $\lambda$-ray is given by 
\begin{align}\label{eq:jacobi:est:ref}
    \begin{split}
        J\left(t_0+\right)&=\rho J\left(t_0-\right), \ \text{and}\ \\
        D_t J\left(t_0+\right)=\rho D_t J\left(t_0-\right)&-\left(\beta\left(\gamma\left(t_0-\right),  \dot{\gamma}\left(t_0-\right)\right)+1\right) \frac{\left\langle J\left(t_0-\right), \nu\right\rangle}{\left\langle I\left(t_0-\right), \nu\right\rangle} \rho D_t I\left(t_0-\right)\\&+\left\langle\dot{\gamma}\left(t_0-\right), \nu\right\rangle^{-1}A J\left(t_0-\right). 
    \end{split}
\end{align}

\end{remark}
\begin{lemma}\label{lm:is:14}
Let $(M,g)$ be a compact Riemannian surface with smooth boundary. Under the assumptions of Lemma \ref{lm:jac:beta:int:1}, a Jacobi field $J$ along a broken $\lambda$-ray satisfies 
\begin{equation}
\begin{split}
      &  |J\left(t_{0}+\right)|^2+|D_{t} J\left(t_{0}+\right)|^2 \\
      & \qquad\le \frac{C}{\left\langle\dot{\gamma}\left(t_0+\right), \nu\right\rangle}\left(\left|J\left(t_0-\right)\right|^2+\left|D_t J\left(t_0-\right)\right|^2+|\dot \gamma (t_0-)|^2+|D_t\dot{\gamma}\left(t_0-\right)|^2  \right),
\end{split}
\end{equation}
at every reflection point $t_0$, where $C$ is a constant depending on $M$, $g$ and $\lambda$.
\end{lemma}
\begin{proof}
From \eqref{eq:jacobi:est:ref} and \eqref{eq:Phi_A_relation}, we have 
\[
|J\left(t_{0}+\right)|=|\rho J\left(t_{0}-\right)|=|J\left(t_{0}-\right)|,
\]
and
\begin{align*}
&|D_t J\left(t_0+\right)| \\
&\qquad\le 
\left|\rho D_{t} J(t_0-)\right|+(\|\beta\|_{L^{\infty}(\pi^{-1}\rR)}+1)\left|\frac{\left\langle J\left(t_0-\right), \nu\right\rangle}{\left\langle I\left(t_0-\right), \nu\right\rangle} \rho  D_t I(t_0-)\right|\\&\qquad\qquad+|\Phi_{\dot{\gamma}(t_0-)}  J(t_0-)|
\\
&\qquad\le |D_tJ\left(t_0-\right) |+(\|\beta\|_{L^{\infty}(\pi^{-1}\rR)}+1)|\left\langle\dot{\gamma}\left(t_0-\right), \nu\right\rangle^{-1}|\left|\left\langle J\left(t_0-\right), \nu\right\rangle   D_t I(t_0-)\right|\\
&\qquad\qquad+|\left\langle\dot{\gamma}\left(t_0-\right), \nu\right\rangle^{-1} A J\left(t_0-\right)|,
\end{align*}
 which implies \begin{align*}
       & |J\left(t_{0}+\right)|^2 
        +|D_{t} J\left(t_{0}+\right)|^2 \\
        &\qquad\le \frac{C}{|\left\langle\dot{\gamma}\left(t_0-\right), \nu\right\rangle|^2}\left(\left|J\left(t_0-\right)\right|^2+\left|D_t J\left(t_0-\right)\right|^2+|D_t\dot{\gamma}\left(t_0-\right)|^2 \right)\\
&\qquad\le          \frac{C}{|\left\langle\dot{\gamma}\left(t_0-\right), \nu\right\rangle|^2}\left(\left|J\left(t_0-\right)\right|^2+\left|D_t J\left(t_0-\right)\right|^2+|\dot{\gamma}\left(t_0-\right)|^2+|D_t\dot{\gamma}\left(t_0-\right)|^2 \right).\qedhere
    \end{align*} 
\end{proof}
\begin{remark}\label{rm:no_refle_bound}
Consider a family of broken $\lambda$-rays on $M$ satisfying $|\langle\dot{\gamma}, \nu\rangle| \geq a$ at each of the reflection points. Due to the compactness of $M$ and the requirement for traversality $|\langle\dot{\gamma}, \nu\rangle| \geq a$, we can assert the existence of a positive real number $l$ that bounds from below by the distance between any two consecutive reflection points for any broken $\lambda$-ray in the set. This provides us with a lower bound on the number of reflections.
If we denote by $N(t)$ the number of reflections of $\gamma$ in the time interval $(0,t)$, then from the preceding discussion it follows that $(N(t)-1)l \le t$, implying $N(t) \le 1 + \frac{t}{l}$.
\end{remark}
\begin{remark}\label{rm:jac:bound}
    Consider a unit-speed $C^1$ curve on the manifold $SM$, given by the mapping $s \in (-\varepsilon, \varepsilon) \mapsto (x_s, v_s)$. Let $\gamma_{x_s, v_s}(t)$ denote a $\lambda$-geodesic such that its initial conditions are $(\gamma_{x_s, v_s}(0), \dot\gamma_{x_s, v_s}(0)) = (x_s, v_s)$. Assume that $\gamma := \gamma_{x_0, v_0}$.
Now, let us examine the Jacobi field $J_s(t) = \partial_s\gamma_{x_s, v_s}(t)$. In this context, we have
    \begin{align*}
        |J_s(0)|^2+\left|D_t J_s(0)\right|^2&=\left|\partial_sx_s\right|^2+\left|D_s v_s\right|^2=1.
    \end{align*}
    Additionally, consider another Jacobi field $I(t) = \dot\gamma_{x, v}(t)$. In this case, we obtain
    \begin{align*}
        |I(0)|^2+|D_tI(0)|^2=|v|^2+|\lambda (x,v) iv |^2\le 1+\|\lambda\|^2
        _{L^\infty(SM)}.
    \end{align*}
\end{remark}
Without any assumption $\lambda\circ\rho =\lambda\beta$ on $\pi^{-1}\rR$, we can still control $|J(t)|^2+\left|D_t J(t)\right|^2$ as follows: 
\begin{corollary}\label{cor:final:jac:est}
 Let $(M,g)$ be a compact Riemannian surface with smooth boundary and $\lambda \in C^\infty(SM)$, and fix a number $a \in (0,1]$. Let $\gamma:[0,\tau]\to M$ be a broken $\lambda$-ray on $M$ such that $|\langle\dot{\gamma}, \nu\rangle| \geq a$ at every reflection point. Then for any variation field $J$ along $\gamma$, we have
\begin{equation}\label{eq:jac:bound:ref:uniform}
|J(t)|^2 + |D_t J(t)|^2 \leq A e^{B t} \left( |J(0)|^2 + |D_t J(0)|^2+ 1+\|\lambda\|^2_{L^{\infty}(SM)} \right)
\end{equation}
for all $t \in [0,\tau]$, where $A$ and $B$ are constants that depend only on $M,g,\lambda$ and $a$.
\end{corollary}
\begin{proof}
Let us assume that there is a reflection at each $t=t_k$, where $k\in \{0,\cdots, N\}$. 
Then for any $t\in (0,t_0)$, using Lemma \ref{lm:Jacobi:est}, we have
\begin{equation}\label{eq:sr:1}
    |J(t)|^2+\left|D_t J(t)\right|^2 \leq e^{C t}\left(|J(0)|^2+\left|D_t J(0)\right|^2\right)
\end{equation} 
and also by considering Jacobi field $I(t)=\dot\gamma(t)$, we have 
\begin{align}\label{eq:sr:2}
|\dot\gamma(t)|^2+|D_t\dot\gamma(t)|^2\le  e^{C t}\left(|\dot\gamma(0)|^2+\left|D_t \dot\gamma(0)\right|^2\right)\le e^{C t}\left(1+\|\lambda\|^2_{L^{\infty}(SM)}\right).
\end{align}
This proves the case for $N=0$.

    Similar to the proof of Lemma \ref{lm:jac:beta:int:1}, we have $\tilde J$ such that \[J(t)=\tilde{J}(t)+\frac{\left\langle J\left(t_0-\right), \nu\right\rangle}{\left\langle I\left(t_0-\right), \nu\right\rangle} I(t).\]
We have $|I\left(t_0+\right)|=|I\left(t_0-\right)|=1$ and
    \begin{align*}
|D_t I\left(t_0+\right) |& =|\lambda\left(\gamma\left(t_0+\right), \dot{\gamma}\left(t_0+\right)\right) i \dot{\gamma}\left(t_0+\right)| \\
& =|\lambda \circ \rho\left(\gamma\left(t_0-\right), \dot{\gamma}\left(t_0-\right)\right)|| i \rho \dot{\gamma}\left(t_0-\right)| \\
&\le \|\lambda\|_{L^{\infty}(\pi^{-1}\rR)}.
\end{align*}
Also we have $|\tilde J(t_0+)|=|\tilde J(t_0-)|$ and from \eqref{eq:int:a12}, we obtain
\begin{align*}
    |D_t\tilde J(t_0+)|&=|\rho D_t \tilde{J}\left(t_0-\right)-\Phi_{\dot{\gamma}\left(t_0-\right)} \tilde{J}\left(t_0-\right)|\\
    &\le | D_t \tilde{J}\left(t_0-\right)|+ \frac{C}{|\left\langle\dot{\gamma}\left(t_0+\right), \nu\right\rangle|}|\tilde{J}\left(t_0-\right)|.
\end{align*}
We assume that for $t\in (0,t_{k-1})$ i.e. before the $k$-th reflection, we have
\begin{equation}\label{eq:J:ma:k}
    |J(t)|^2+\left|D_t J(t)\right|^2 \leq \frac{e^{k C t}  C_1^{k-1}}{a^{k-1}}\left(|J(0)|^2+\left|D_t J(0)\right|^2+(k-1)(1+\|\lambda\|^2_{L^{\infty}(SM)})\right).
\end{equation}
Also 
\begin{align}\label{eq:I:ma:k}
    |\dot \gamma(t)|^2+\left|D_t\dot \gamma (t)\right|^2 \leq e^{k C t} \frac{C_1^{k-1}}{a^{k-1}}\left(1+\|\lambda\|^2_{L^{\infty}(SM)}\right).
\end{align}
We now prove the estimate for any $t\in (0,t_k)$, that is before the $(k+1)$-th reflection. Using the assumption that $|\langle\dot{\gamma}, \nu\rangle| \geq a$ where $a\in (0,1]$ and Lemma \ref{lm:is:14}, we can deduce that at each reflection point, we have
\[
 \left|J\left(t_k+\right)\right|^2+\left|D_t J\left(t_k+\right)\right|^2 \leq \frac{C_1}{a}\left(\left|J\left(t_k-\right)\right|^2+\left|D_t J\left(t_k-\right)\right|^2+\left| \dot{\gamma}\left(t_k-\right)\right|^2+\left|D_t \dot{\gamma}\left(t_k-\right)\right|^2\right)\label{eq:J:reflection}   
\]
   and 
  \begin{align}
       \left| \dot{\gamma}\left(t_k+\right)\right|^2+\left|D_t \dot{\gamma}\left(t_k+\right)\right|^2
       &\le \frac{C_1}{a}\left(\left| \dot{\gamma}\left(t_k-\right)\right|^2+\left|D_t \dot{\gamma}\left(t_k-\right)\right|^2\right). \label{eq:I:reflection}
  \end{align}
\noindent
Combining \eqref{eq:J:ma:k} with \eqref{eq:I:ma:k} and Lemma \ref{lm:is:14}, we have
\begin{align*}
     &\left|J\left(t\right)\right|^2+\left|D_t J\left(t\right)\right|^2 \\
     &\qquad\leq e^{C t}\left(|J(t_{k-1}+)|^2+\left|D_t J(t_{k-1}+)\right|^2\right)\\
     &\qquad\le e^{C t}\frac{C_1}{a}\left(\left|J\left(t_{k-1}-\right)\right|^2+\left|D_t J\left(t_{k-1}-\right)\right|^2+\left|\dot{\gamma}\left(t_{k-1}-\right)\right|^2+\left|D_t \dot{\gamma}\left(t_{k-1}-\right)\right|^2\right)\\
     &\qquad\le e^{(k+1)C t}\frac{C_1^k}{a^k}\left(|J(0)|^2+\left|D_t J(0)\right|^2+(k-1)(1+\|\lambda\|^2
        _{L^\infty(SM)}) \right)\\
     &\qquad\qquad + e^{(k+1)C t}\frac{C_1^k}{a^k}\left(1+\|\lambda\|^2
        _{L^\infty(SM)}\right)\\
     &\qquad=e^{(k+1)C t}\frac{C_1^k}{a^k}\left(|J(0)|^2+\left|D_t J(0)\right|^2+k(1+\|\lambda\|^2
        _{L^\infty(SM)}) \right).
\end{align*}
and 
\begin{align*}
    |\dot{\gamma}(t)|^2+\left|D_t \dot{\gamma}(t)\right|^2&\le e^{Ct} \left(\left|\dot\gamma\left(t_{k-1}+\right)\right|^2+\left|D_t \dot\gamma\left(t_{k-1}+\right)\right|^2\right)\\
    &\le e^{Ct}\frac{C_1}{a} \left(\left|\dot\gamma\left(t_{k-1}-\right)\right|^2+\left|D_t \dot\gamma\left(t_{k-1}-\right)\right|^2\right)\\
    &\le e^{(k+1)C t}\frac{C_1^k}{a^k}\left(1+\|\lambda\|^2
        _{L^\infty(SM)}\right).
\end{align*}
Note that 
\begin{align*}
     \frac{C_1^k}{a^k}&\left(|J(0)|^2+\left|D_t J(0)\right|^2+k(1+\|\lambda\|^2
        _{L^\infty(SM)})\right)\\ &\qquad\qquad \le  \frac{(2C_1)^k}{a^k}\left(|J(0)|^2+\left|D_t J(0)\right|^2+1+\|\lambda\|^2
        _{L^\infty(SM)}\right).
\end{align*}
From this analysis one can see that there is constant $A=e^{Ct}\frac{2C_1}{a}$, such that at each reflection point $|J(t)|^2+|D_tJ(t)|^2$ increases by the factor $A$.
Now consider the interval $(0,t)$ where any broken ray has less than $1+t/l$ number of reflections by Remark \ref{rm:no_refle_bound}. We may conclude the estimate
\begin{align*}
     |J(t)|^2+\left|D_t J(t)\right|^2&\le A^{N(t)}e^{Ct}\left(|J(0)|^2+\left|D_t J(0)\right|^2+1+\|\lambda\|^2
        _{L^\infty(SM)}\right)\\
     &= e^{\left(1+\frac{t}{l}\right)\log A+Ct}\left(|J(0)|^2+\left|D_t J(0)\right|^2+1+\|\lambda\|^2
        _{L^\infty(SM)}\right)\\
&=Ae^{Bt}\left(|J(0)|^2+\left|D_t J(0)\right|^2+1+\|\lambda\|^2
        _{L^\infty(SM)}\right),
\end{align*}
where $B=\frac{\log A}{l}+C$. Note that here, $A$ and $B$ are constants that depend only on $M, g, \lambda$, and $a$, which proves the lemma.
\end{proof}

\subsection{Regularity of solutions to the transport equation}\label{sec:regularity-of-solutions}
In \cite[Lemma 3]{Ilmavirta:Paternain:2022}, a regularity result is proven for the primitive function corresponding to $f\in C^2(SM)$, which is simply the solution to the transport equation. This result shows that the primitive function is twice continuously differentiable in the interior of $SM$ and Lipschitz continuous in $SM$. In this subsection, we extend this result to the solution of the transport equation associated with broken $\lambda$-rays.
We first aim at establishing the regularity result for both forward and backward exit times. In \cite[Lemma 3.2.3]{GIP2D}, the regularity of exit time has been demonstrated in the context of usual geodesics. Following a similar approach, we extend this result to the case of the $\lambda$-geodesic.
\begin{lemma}\label{lm:tau:smooth}
Let $(M, g)$ be a compact nontrapping Riemannian surface with strictly $\lambda$-convex boundary where $\lambda\in C^{\infty}(SM)$. Then $\tau$ and $\tau^{-}$ are continuous on $SM$ and smooth on $SM\setminus \partial_0SM$. 
\end{lemma}
\begin{proof}
We establish the result for the forward exit time only as the proof is identical for $\tau^{-}$.
     Let $(N, g)$ be a closed extension of $(M, g)$ (cf. \cite[Lemma 3.1.8]{GIP2D}). Define $\rho$ as a boundary defining function on $M$. Consider a $\lambda$-geodesic $\gamma$ with initial conditions $\gamma(0) = x$ and $\dot\gamma(0) = v$. We now analyze the function $\rho(\gamma_{x, v}(t))$ for $(x, v) \in SM \setminus \partial_0 SM$, where $\rho$ is a boundary defining function on $M$. 
     
First, we show that $\tau$ is continuous on $S M$. Let $\left(x_0, v_0\right) \in \operatorname{int} S M$ with $\tau\left(x_0, v_0\right)>0$. We can choose $\epsilon>0$ such that $\inf _{t \in\left[0, \tau\left(x_0, v_0\right)-\epsilon\right]} \rho\left(\gamma_{x_0, v_0}(t)\right)>0$. This condition ensures that $\inf _{t \in\left[0, \tau\left(x_0, v_0\right)-\epsilon\right]} \rho\left(\gamma_{x, v}(t)\right)>0$ for $(x, v)$ close to $\left(x_0, v_0\right)$. Consequently, $\gamma_{x, v}(t) \in M$ for $0<t<\tau\left(x_0, v_0\right)-\epsilon$ and $(x, v)$ near $\left(x_0, v_0\right)$, which implies $\tau(x, v) \geq \tau\left(x_0, v_0\right)-\epsilon$ for $(x, v)$ close to $\left(x_0, v_0\right)$, i.e., 
\begin{equation}\label{eq:tau:1:c}
    \tau(x, v)-\tau\left(x_0, v_0\right)>-\epsilon.
\end{equation}
Given that $\partial M$ is strictly $\lambda$-convex, we find that $\rho\left(\gamma_{x_0, v_0}\left(\tau\left(x_0, v_0\right)+\epsilon\right)\right)<0$ for some small $\epsilon>0$. This condition ensures that $\rho\left(\gamma_{x, v}\left(\tau\left(x_0, v_0\right)+\epsilon\right)\right)<0$ for $(x, v)$ close to $\left(x_0, v_0\right)$. Thus, $\tau(x, v) \leq \tau\left(x_0, v_0\right)+\epsilon$ for $(x, v)$ near $\left(x_0, v_0\right)$, implying $\tau(x, v)-\tau\left(x_0, v_0\right)<\epsilon$. 
\begin{equation}\label{eq:tau:2:c}
    \tau(x, v)-\tau\left(x_0, v_0\right)<\epsilon.
\end{equation}
Therefore, from \eqref{eq:tau:1:c} and \eqref{eq:tau:2:c}, we can find $\epsilon>0$ such that $\left|\tau(x, v)-\tau\left(x_0, v_0\right)\right|<\epsilon$ for $(x, v)$ in the neighborhood of $\left(x_0, v_0\right)$. A similar argument applies for $\left(x_0, v_0\right) \in \partial S M$, in which case, $\tau\left(x_0, v_0\right)=0$.
     
 Next, we show that $\tau$ is smooth in $SM\setminus \partial_0SM$. Similar to the proof of \cite[Lemma 3.2.3]{GIP2D}, we have $h: S N \times \mathbb{R} \rightarrow \mathbb{R}, $ such that $ h(x, v, t):=\rho\left(\gamma_{x, v}(t)\right)$ and
     \begin{align*}
       \frac{\partial h}{\partial t}(x, v, t)=  \frac{\partial}{\partial t}\rho\left(\gamma_{x, v}(t)\right)=\left.\left\langle\nabla \rho\left(\gamma_{x,v}(t)\right), \dot{\gamma}_{x,v}(t)\right)\right\rangle.
     \end{align*}
     Note that $\gamma_{x,v}(\tau_{x,v})\in \partial M$. This implies that the tangent vector $\dot\gamma_{x,v}(\tau_{x,v})$ must lie in $\partial_- SM$; otherwise if $\dot\gamma_{x,v}(\tau_{x,v})$ were not in $\partial_- SM$, the geodesic $\gamma_{x,v}$ could be extended beyond the point $\gamma_{x,v}(\tau_{x,v})$, contradicting the fact that $\gamma_{x,v}(\tau_{x,v})$ is the final point on $M$.
    By strict $\lambda$-convexity,  one must have $\dot\gamma_{x,v}(\tau_{x,v})\notin \partial_0SM$. 
Since $\dot\gamma_{x,v}(\tau_{x,v})\in \partial SM\setminus \partial_0 SM$, i.e. $\langle \dot\gamma_{x,v}(\tau_{x,v}),\nu\rangle<0$ and 
\begin{align*}
    \frac{\partial h}{\partial t}(x, v, \tau^+_{x, v})=\langle\nabla \rho(\gamma_{x,v}(t)), \dot{\gamma}_{x,v}(t))\rangle |_{\tau^{+}_{x,v}} = \langle \nu,\dot\gamma_{x,v}(\tau_{x,v}) \rangle<0,
\end{align*}
it follows from the definition of boundary defining function that
$h(x,v,\tau_{x,v})=0$ and $h$ is smooth. Finally, by the implicit function theorem, we conclude that $\tau$ is smooth in $SM\setminus \partial_0SM$.
\end{proof}
\begin{remark}
    Similar to the case of broken rays (cf. \cite[p. 399]{Ilmavirta:Salo:2016}), using Lemma \ref{lm:tau:smooth}, $\tau$ and $\tau^-$ are smooth near any point $(x,v)$ such that the broken $\lambda_{x,v}$-ray reflects and exits transversely.
\end{remark}
\begin{remark}
Let $\sigma:(-\epsilon,\epsilon)\to SM$ be a smooth curve such that $\sigma(s)=(x_s,v_s)$. Note that the function $h$ (same as defined in Lemma \ref{lm:tau:smooth}), satisfies the property that $h(x_s,v_s,\tau(x_s,v_s))=0$ and $\frac{\partial h}{\partial t}\left(x_s, v_s, \tau_{x_s, v_s}\right)\ne 0$. By the implicit function theorem, we have 
\begin{align}\label{eq:tau:derivative}
   \begin{split}
        \partial_s\tau_{x_s, v_s}&= -\left[ \frac{d}{dt}h(x_s,v_s,\tau_{x_s,v_s})\right]^{-1}\left[ \partial_sh(x_s,v_s,\tau_{x_s,v_s})\right]\\
    &=-\frac{\left.\left\langle\nabla \rho\left(\gamma_{x_s, v_s}(t)\right), \partial_s{\gamma}_{x_s, v_s}(t)\right)\right\rangle\left.\right|_{t=\tau_{x_s, v_s}^{}}}{\left.\left\langle\nabla \rho\left(\gamma_{x_s, v_s}(t)\right), \dot{\gamma}_{x_s, v_s}(t)\right)\right\rangle\left.\right|_{t=\tau_{x_s, v_s}^{}}}\\
    &=-\frac{\left.\left\langle\nu, \partial_s{\gamma}_{x_s, v_s}(t)\right)\right\rangle\left.\right|_{t=\tau_{x_s, v_s}^{}}}{\left.\left\langle\nu, \dot{\gamma}_{x_s, v_s}(t)\right)\right\rangle\left.\right|_{t=\tau_{x_s, v_s}}}.
   \end{split}
\end{align}
\end{remark}
\begin{definition}
Let $(M,g)$ be a Riemannian surface and $\lambda\in C^{\infty}(SM)$.
  Let us denote the interior $\lambda$-scattering relation by $\tilde \alpha : SM \rightarrow \partial SM$. Given any point and direction $(x, v) \in SM$, we map it via $\tilde\alpha$ to the first intersection point and direction of the $\lambda$-geodesic $\gamma_{x, v}$ with the boundary $\partial M$ (i.e. either $\E$ or $\rR$).
\end{definition}
\begin{remark}
    Note that in the case a $\lambda$-geodesic $\gamma_{x,v}$ hits $\partial M$ non-tangentially, then in a neighborhood of $(x,v)$ the map $\tilde\alpha$ is smooth.
\end{remark}
\begin{remark}\label{rm:smooth_dependence}
    When a broken $\lambda$-ray $\gamma_{x,v}$ hits non-tangentially to $\rR$ (possibly multiple times) and reach a point on $\E$ transversally (by strict convexity), then there is a smooth dependence for the end point on $\E$ of the broken $\lambda$-ray on its initial data $(x,v)$. This follows from the smoothness of $\tilde\alpha$ and $\rho$ map.
\end{remark}
\begin{lemma}\label{lm:reg:u}
Let $(M,g,\lambda,\E)$ be an admissible dynamical system. If $f\in C^2(SM)$ satisfies $If=0$, then the primitive function $u$ solving \eqref{eq:trans:1} has the regularity $u\in C^2(\operatorname{int}(SM)) \cap \operatorname{Lip}(SM)$.
\end{lemma}
It is clear that $u$ solves \eqref{eq:trans:1}. Hence, we split the proof into two cases for the regularity of $u$.
 \begin{proof}[Proof of $u\in C^2(\operatorname{int}SM)$]\renewcommand{\qedsymbol}{}Let $(x,v)\in \operatorname{int}(SM)$. From the admissibility condition, $\gamma_{x,v}(t)$ or $\gamma^-_{x,-v}(t)$ has no tangential reflections.  From Lemma \ref{dual_tr_lem}, we have 
  \begin{equation}\label{eq:u:pm}
     u(x,v)=-u^-(x,-v)
 \end{equation}
 where $u^-$ is the solution to the dual transport equation. Hence it suffices to show that either $u$ or $u^-$ is $C^2$ at $(x,v)$ or $(x,-v)$ respectively. Without loss of generality, we may assume $\gamma_{x,v}(t)$ has no tangential reflections. Now, for some $N =0,1,2,\dots$, we have
\begin{align}\label{eq:sum:k:u}
     u(x,v)=\sum_{k=0}^N u_{k}(x,v)
 \end{align}
 where 
 \begin{align*}
u_k(x,v)=\int_{\tau_k}^{\tau_{k+1}}f(\phi_t(x,v))dt
 \end{align*}
 and $\gamma_{x,v}$ has reflections at $\tau_1,\cdots \tau_N$ with $\tau_0=0$, $\tau_{N+1}=\tau_{x,v}$.  
Since the broken $\lambda$-ray hits $\rR$ transversely, $\tau_k(x,v)$ are smooth in some neighborhood of $(x,v)$, $\lambda$-geodesic flow is smooth and $f\in C^2(SM)$, we have that all $u_k(x,v)$ are $C^2$ functions in some neighborhood of the point $(x,v)$. 
 As each $u_k$ is a $C^2$ function at the point $(x,v)$, it follows from \eqref{eq:sum:k:u} that the function $u$ is also $C^2$ at $(x,v)$.
\end{proof}
\begin{proof}[Proof of $u\in \operatorname{Lip}(SM)$] If we show that first order derivatives $u$ are uniformly bounded in int$(SM)$, then this implies $u$ is Lipschitz. To show this, similar to \cite[p. 1283]{Ilmavirta:Paternain:2022}, we consider a $C^1$ unit speed curve $(-\varepsilon, \varepsilon) \ni s \mapsto\left(x_s, v_s\right) \in \operatorname{int} S M$  with $\left(x_0, v_0\right)=(x, v)$. Using Lemma \ref{dual_tr_lem} again, we can assume without loss of generality that $\gamma_{x,v}$ have no tangential reflections and $|\langle \dot\gamma,\nu\rangle |\geq a$. Now
\begin{align}
     \begin{split}
      \partial_s u\left(x_s, v_s\right)&=  f\left(\gamma_{x_s, v_s}\left(\tau_{x_s, v_s}\right), \dot{\gamma}_{x_s, v_s}\left(\tau_{x_s, v_s}\right)\right) \partial_s\tau_{x_s, v_s} \\ & \qquad+\int_0^{\tau_{x_s, v_s}} \partial_s f\left(\gamma_{x_s, v_s}(t), \dot{\gamma}_{x_s, v_s}(t)\right) \mathrm{d} t.
     \end{split}
 \end{align}
Let us start by examining the integral term
\begin{align*}
\left.\partial_s f\left(\gamma_{x_s, v_s}(t), \dot{\gamma}_{x_s, v_s}(t)\right)\right|_{s=0} &= \langle (\partial_s\gamma_{x_s, v_s},D_s\dot{\gamma}_{x_s, v_s})|_{s=0}, \nabla_{SM} f (\gamma(t),\dot\gamma(t))\rangle = \langle (J,\dot J ),\nabla_{SM}f \rangle
\end{align*}
where $J := \partial_s \gamma_{x_s,v_s}$ is a broken $\lambda$-Jacobi field.
Since $\gamma_{x,v}$ contains no reflections with $|\langle \dot\gamma,\nu\rangle |<a$, it follows from Lemma \ref{cor:final:jac:est} and Remark \ref{rm:jac:bound} that there exists a uniform $C_1>0$ such that $|J|^2+|\dot J|^2\le C_1$ holds for all $(x,v) \in \inte(SM)$. This implies 
\begin{equation}\label{eq:thefirstpart-of-estimate}
  \left| \left. \int_0^{\tau_{x_s, v_s}} \partial_s f\left(\gamma_{x_s, v_s}(t), \dot{\gamma}_{x_s, v_s}(t)\right) \mathrm{d} t\right|_{s=0}\right|\le L C_1^{1/2} \|\nabla_{SM}f \|_{L^{\infty}(SM)}.
\end{equation}
 We now focus on the boundary term. By taking $s=0$ in \eqref{eq:tau:derivative}, we have 
 \begin{equation}\label{eq:in:b1}
     \left.\partial_s\tau_{x_s, v_s}\right|_{s=0}=\left.-\frac{\langle J, \nu\rangle}{\left\langle\dot{\gamma}_{x, v}, \nu\right\rangle}\right|_{t=\tau_{x, v}}.
 \end{equation}
 From the proof of Lemma \ref{lm:tau:smooth}, for any $(x,v)\in \inte SM$, $\left\langle\dot{\gamma}_{x, v}\left(\tau_{x, v}\right), \nu\right\rangle<0$.
 From the expression \eqref{eq:in:b1}, we need to consider two cases: $|\langle \dot\gamma_{x,v}(\tau_{x,v}),\nu \rangle|<b$ and $|\langle \dot\gamma_{x,v}(\tau_{x,v}),\nu \rangle|\ge b$ for some small enough $b>0$. We choose $b$ such that whenever $|\langle \dot\gamma_{x,v}(\tau_{x,v}),\nu \rangle|<b$, then the corresponding broken geodesic has no reflections for any $(x,v) \in \inte SM$. A choice of a very small parameter $b>0$ splits $\inte(SM)$ into two sets corresponding to the broken $\lambda$-rays to short ones which are almost tangential to $\E$ and all other broken rays since $\E$ strictly $\lambda$-convex.  \\
 \emph{The case $|\langle \dot\gamma_{x,v}(\tau_{x,v}),\nu \rangle|\ge b$.}
 Using the strict $\lambda$-convexity of $\E$, it follows that $\lambda$-geodesics intersect $\E$ transversely. This implies that $\left.\partial_s\tau_{x_s, v_s}\right|_{s=0}$ is uniformly bounded by $C_1^{1/2}/b$, which indeed shows that \begin{equation}\label{eq:first-bounded-estimate}
    \left|\left.f\left(\gamma_{x_s, v_s}\left(\tau_{x_s, v_s}\right), \dot{\gamma}_{x_s, v_s}\left(\tau_{x_s, v_s}\right)\right) \partial_s\tau_{x_s, v_s} \right|_{s=0}\right|\le \frac{C_1^{1/2}}{b} \|f\|_{L^{\infty}(SM)}.
\end{equation}\\
 \emph{The case $|\langle \dot\gamma_{x,v}(\tau_{x,v}),\nu \rangle|<b$.} By the choice of $b$, we have that $\gamma_{x,v}$ never reaches $\rR$ and corresponds to a short $\lambda$-geodesic almost tangential to $\partial M$. Since the broken ray transform vanishes, we have $f(y,w)=0$ holds for $y\in \E$ and $w\in S_y\E$.

Write $(y_s,w_s)=\left(\gamma_{x_s, v_s}\left(\tau_{x_s, v_s}\right), \dot{\gamma}_{x_s, v_s}\left(\tau_{x_s, v_s}\right)\right)$. Since $f$ is Lipschitz, for any $w\in S_{y_s}\E $, we have
\begin{equation}\label{eq:lip1}
   |f(y_s,w_s)|=|f(y_s,w_s)-f(y_s,w)|\le C|w_s-w|. 
\end{equation}
Let us express $w_s$ in terms of $\nu$ and $w$ where we choose the orientation so that $w\in S_{y_s}\E$ and $\langle w,w_s\rangle\ge 0$. Now
\begin{equation*}
w_s = \langle w_s, \nu \rangle \nu + \langle w_s, w \rangle w= \langle w_s, \nu \rangle \nu  + \sqrt{1-\langle w_s, \nu \rangle ^2} w.
\end{equation*}
 This implies 
\begin{align*}
    |w_s-w|^2\le \left(\sqrt{1-\langle w_s, \nu \rangle ^2}-1\right)^2+\langle w_s, \nu \rangle^2. 
\end{align*}
Observe that 
\begin{align*}
    \left(\sqrt{1-\left\langle w_s, \nu\right\rangle^2}-1\right)^2\le \left\langle w_s, \nu\right\rangle^2 \iff 1-\left\langle w_s, \nu\right\rangle^2\le \sqrt{1-\left\langle w_s, \nu\right\rangle^2}.
\end{align*}
But in case $-1\le \left\langle w_s, \nu\right\rangle\le 1$, we have $1-\left\langle w_s, \nu\right\rangle^2\le \sqrt{1-\left\langle w_s, \nu\right\rangle^2}$, which proves that 
\begin{equation}\label{eq:lip2}
    \left|w_s-w\right|^2\le 2 \left\langle w_s, \nu\right\rangle^2.
\end{equation}
From \eqref{eq:lip1} and \eqref{eq:lip2}, we have
\begin{equation}\label{eq:in:b0}
    \left|f\left(y_s, w_s\right)\right|\le \sqrt{2} \left|\langle w_s, \nu\rangle\right|.
\end{equation}
Using \eqref{eq:in:b1}, we can write
\begin{equation}\label{eq:reg:int:2}
    \left.\partial_s\tau_{x_s, v_s}\right|_{s=0}\le \frac{C_1^{1/2}}{|\langle \nu, w_0\rangle |}.
\end{equation}
From \eqref{eq:first-bounded-estimate}, \eqref{eq:in:b0} and \eqref{eq:reg:int:2}, we conclude for any $(x,v) \in \inte(SM)$ that
\begin{equation}\label{eq:final-bdry-estimate}
    \left|\left.f\left(\gamma_{x_s, v_s}\left(\tau_{x_s, v_s}\right), \dot{\gamma}_{x_s, v_s}\left(\tau_{x_s, v_s}\right)\right) \partial_s\tau_{x_s, v_s} \right|_{s=0}\right|\le C \|f\|_{L^{\infty}(SM)}
\end{equation}
for some $C>0$.
It follows from \eqref{eq:thefirstpart-of-estimate} and \eqref{eq:final-bdry-estimate} that $u \in W^{1,\infty}(SM)$. We conclude $u \in \mathrm{Lip}(SM)$.
\end{proof}

\section{Uniqueness for scalar functions and \texorpdfstring{$1$}{1}-forms}
\subsection{Revisiting the boundary terms in the Pestov identity}
The primitive function corresponding to $f$, denoted by $u := u^{f}$, is defined as  
\begin{equation}\label{premitiveF}
    u^{f}(x,v)= \int_{0}^{\tau(x,v)}f(\phi_t(x,v))dt
\end{equation}
 where $\phi_t$ is the broken $\lambda$-geodesic flow. 
 
In the following lemma, we provide a simplified form of the boundary term $\nabla_{T,\lambda}$ appearing in the Pestov identity \eqref{eq:Pestov3.1} in terms of the odd and even components of $u$ and the magnetic signed curvature. In \cite[Lemma 9]{Ilmavirta:Salo:2016}, a similar identity has been proved for the broken geodesic flow (i.e. when $\lambda=0$). In particular, they showed that
\begin{equation}\label{eq:lm:9:is}
\left(\nabla_T u, V u\right)_{\partial SM}=\left(\nabla_T u_e, V u_o\right)_{\partial SM}+\left(\nabla_T u_o, V u_e\right)_{\partial SM} -(\kappa V u, V u)_{\partial S M}
\end{equation}
where $\nabla_T=\nabla_{T,0}$, $\kappa := - \langle D_t T, \nu \rangle_g$ is the signed curvature of $\partial M$, $u_{e}$ and $u_{o}$ are the even and odd components of $u|_{\partial SM}$ with respect to the reflection $\rho$ and $u$ is a primitive function.
We now aim to prove the following generalization of \cite[Lemma 9]{Ilmavirta:Salo:2016} to the case of broken $\lambda$-geodesic flows. 
\begin{lemma}\label{lm:uf:eo:g}
Let $(M, g)$ be a compact Riemannian surface with smooth boundary and $\lambda \in C^\infty(SM)$.  If $u \in C^2(SM)$, then
\begin{align}\label{eq:odd:even:com:1}
\begin{split}
    \left(\nabla_{T, \lambda} u, V u\right)_{\partial SM}&= \left(\nabla_T u_e, V u_o\right)_{\partial S M}+\left(\nabla_T u_o, V u_e\right)_{\partial S M}-(\kappa V u, V u)_{\partial S M}\\&\qquad-\left(V(\mu)(\lambda V u)_e+\mu(V(\lambda) V u)_o, Vu_o\right)_{\partial S M}\\&\qquad\qquad -\left(V(\mu)(\lambda V u)_o+\mu(V(\lambda) V u)_e, Vu_e  \right)_{\partial S M}.\qedhere
\end{split}
\end{align}
\end{lemma}
\begin{proof}
 We denote $\nabla_{T, \lambda}u=\nabla_Tu+Lu$ where $Lu:=-\langle v_{\perp},\nu\rangle \lambda V u-\langle v, \nu\rangle(V(\lambda)) Vu.$ Note that from \cite[p. 119]{GIP2D}, we have $\mu(x,v)=\langle v,\nu(x) \rangle$ and $V(\mu)(x,v)=\left\langle v_{\perp}, \nu\right\rangle$. This implies $Lu = - V(\lambda \mu) Vu$.  From \cite[p. 391]{Ilmavirta:Salo:2016}, we have $\left(\rho^* V u\right)=-V\left(\rho^* u\right)$.
We compute 
   \begin{align}
\rho^*(\mu(x,v) )&= \langle (v-2\langle v,\nu \rangle \nu ),\nu\rangle=-\langle v,\nu \rangle=-\mu(x,v)\label{eq:eo:11}\\
    \rho^*(V(\lambda \mu)(x,v))&=  -V(\rho^*(\lambda (x,v)\mu(x,v) )). \label{eq:eo:12}
\end{align}
From \eqref{eq:eo:11} and \eqref{eq:eo:12}, we have
\begin{align*}
    (Lu)_e(x,v)&=\frac{Lu(x,v)+\rho^*Lu(x,v)}{2}\\
   & =\frac{-\mu V(\lambda ) Vu-\lambda V(\mu )Vu-\rho^*(\mu V(\lambda)+V(\mu)\lambda)\rho^*Vu}{2}\\
   &= \frac{-\mu V(\lambda ) Vu-\lambda V(\mu )Vu+\mu\rho^*( V(\lambda)Vu)-V(\mu)\rho^*(\lambda Vu)}{2}\\
   &=-V(\mu)(\lambda V u)_e-\mu (V(\lambda) V u)_o
\end{align*}
and 
\begin{align*}
     (Lu)_o(x,v)&=\frac{Lu(x,v)-\rho^*Lu(x,v)}{2}\\
   & =\frac{-\mu V(\lambda ) Vu-\lambda V(\mu )Vu+\rho^*(\mu V(\lambda)+V(\mu)\lambda)\rho^*Vu}{2}\\
   &= \frac{-\mu V(\lambda ) Vu-\lambda V(\mu )Vu-\mu\rho^*( V(\lambda)Vu)+V(\mu)\rho^*(\lambda Vu)}{2}\\
     &=-V(\mu)(\lambda V u)_o-\mu(V(\lambda) V u)_e.
\end{align*}
Since $\rho$ is an isometry on $S_x$ for each $x \in \partial M$ (cf. Remark \ref{rm:odd:even}), we obtain
\begin{align*}
    &\left(L u, V u\right)_{\partial S M}= \left((L u)_e, (V u)_e\right)_{\partial S M}+\left((L u)_o, (V u)_o\right)_{\partial S M}\\
    &=\left(-V(\mu)(\lambda V u)_e-\mu(V(\lambda) V u)_o, Vu_o\right)_{\partial S M} + \left(-V(\mu)(\lambda V u)_o-\mu(V(\lambda) V u)_e, Vu_e  \right)_{\partial S M}.
\end{align*}
Combining with \eqref{eq:lm:9:is}, we have
\begin{align}
\begin{split}
   & \left(\nabla_{T, \lambda} u, V u\right)_{\partial SM}= \left(\nabla_T u_e, V u_o\right)_{\partial S M}+\left(\nabla_T u_o, V u_e\right)_{\partial S M}-(\kappa V u, V u)_{\partial S M}\\&\quad-\left(V(\mu)(\lambda V u)_e+\mu(V(\lambda) V u)_o, Vu_o\right)_{\partial S M} -\left(V(\mu)(\lambda V u)_o+\mu(V(\lambda) V u)_e, Vu_e  \right)_{\partial S M}.\qedhere
\end{split}
\end{align}
\end{proof}

\begin{remark}\label{rm:odd:even}
   Notice that $\rho$ is an isometry on $S_x$ for each $x \in \partial M$.
   The even and odd parts of $u$ are denoted by $u_e$ and $u_0$ respectively with respect to the isometry $\rho$. Similarly, $v_e$ and $v_o$ stands for the even and odd parts of $v$ respectively with respect to the isometry $\rho$. Then 
    \begin{align*}
        (u_e,v_o)_{\partial SM}&=\left(\frac{u+u\circ \rho}{2},\frac{v-v\circ \rho}{2} \right)_{\partial SM}\\&=\frac{1}{4}\left\{(u,v)_{\partial SM}+(u\circ \rho,v)_{\partial SM}-(u,v\circ \rho)_{\partial SM}-(u\circ \rho,v\circ \rho)_{\partial SM}\right\}\\&
        =\frac{1}{4}\left\{(u\circ \rho,v)_{\partial SM}-(u,v\circ \rho)_{\partial SM}\right\},
    \end{align*}
    where we used the fact that $\rho$ is an isometry. Similarly,
    \begin{align*}
(u_o,v_e)_{\partial SM}&=\left(\frac{u-u\circ \rho}{2},\frac{v+v\circ \rho}{2} \right)_{\partial SM}\\&=\frac{1}{4}\left\{(u,v)_{\partial SM}-(u\circ \rho,v)_{\partial SM}+(u,v\circ \rho)_{\partial SM}-(u\circ \rho,v\circ \rho)_{\partial SM}\right\}\\&
        =\frac{1}{4}\left\{-(u\circ \rho,v)_{\partial SM}+(u,v\circ \rho)_{\partial SM}\right\}.
    \end{align*}
    This implies 
    \begin{align*}
        (u_e,v_o)_{\partial SM}+(u_o,v_e)_{\partial SM}=0,
    \end{align*}
    and in particular, we have 
    \begin{align*}
        (u,v)_{\partial SM}=(u_e+u_o,v_e+v_o)_{\partial SM}=(u_e,v_e)_{\partial SM}+(u_o,v_o)_{\partial SM}.
    \end{align*}
\end{remark}
\begin{corollary}\label{cor:odd:even:simply}
    Let $(M, g)$ be a compact Riemannian surface with smooth boundary and $\lambda \in C^{\infty}(S M)$. If $u \in C^2(S M)$ and $u=u\circ \rho$ on $\partial SM$, then 
    \begin{equation}\label{eq:reduced:even:odd}
      \left(\nabla_{T, \lambda} u, V u\right)_{\partial S M}=-((\kappa_{\lambda}+\eta_{\lambda})_e V u, V u)_{\partial S M},
    \end{equation}
    where $\kappa_\lambda(x, v)=\kappa-\langle\nu(x), \lambda(x, v) i v\rangle$ and $\eta_\lambda(x, v)=\langle V(\lambda)(x, v) v, \nu\rangle$.
\end{corollary}
\begin{proof}
 We have
\begin{align}
      \rho^*(\lambda (x,v) V(u )(x,v) )&=-(\rho^*\lambda) V(\rho^*u)(x,v), \label{eq:eo:13}\\
      \rho^*(V(\lambda)(x,v) Vu(x,v) )&= V(\rho^*\lambda)V(\rho^*u)(x,v).\label{eq:eo:14}
\end{align}
By the assumption on $u$, we have $u_e=u$ on $\partial SM$ and $u_o=0$ on $\partial SM$.
   Combining \eqref{eq:eo:13} and \eqref{eq:eo:14}, we obtain
    \begin{align}
        (\lambda V u)_e(x,v)&= \frac{\lambda V u (x,v)-\left(\rho^* \lambda\right) V\left(\rho^* u\right)(x, v)}{2}= \lambda_o  Vu_e(x,v)\\
        (\lambda V u)_o(x,v)&= \frac{\lambda V u (x,v)+\left(\rho^* \lambda\right) V\left(\rho^* u\right)(x, v)}{2}= \lambda_e Vu_e(x,v)\\
        (V(\lambda) V u)_e&=\frac{V(\lambda)(x, v) V u(x, v)+V\left(\rho^* \lambda\right) V\left(\rho^* u\right)(x, v)}{2}= V(\lambda_e) Vu_e (x,v)\\ (V(\lambda) V u)_o&=\frac{V(\lambda)(x, v) V u(x, v)-V\left(\rho^* \lambda\right) V\left(\rho^* u\right)(x, v)}{2}= V(\lambda_o) Vu_e (x,v)
    \end{align}
   From \eqref{eq:odd:even:com:1} and Lemma \ref{lm:property:rho:kappa:eta},  we have
    \begin{align*}
        \left(\nabla_{T, \lambda} u, V u\right)_{\partial S M}&=-(\kappa V u, V u)_{\partial S M} -\left(V(\mu)\lambda_e V u_e+\mu V\left(\lambda_e\right) V u_e, V u_e\right)_{\partial SM}\\&=-(\kappa V u, V u)_{\partial S M} +\left(-\left\langle v_{\perp}, \nu\right\rangle \lambda_e V u_e-\langle v, \nu\rangle V\left(\lambda_e\right) V u_e, V u_e\right)_{\partial SM}\\
        &=-(\kappa V u, V u)_{\partial S M} +\left(\left\langle \lambda_e iv, \nu\right\rangle  V u-\langle V\left(\lambda_e\right) v, \nu\rangle  V u, V u\right)_{\partial SM}\\
        &=-\left(\left(\kappa_\lambda+\eta_\lambda\right)_e V u, V u\right)_{\partial S M}.\qedhere
    \end{align*}
\end{proof}

\noindent Proposition \ref{lm:pestov:g} and Corollary \ref{cor:odd:even:simply} now lead to the Pestov identity 
\begin{equation}\label{eq:pestov:1}
\begin{split}
\|P u\|_{SM}^2&=\|\widetilde{P} u\|_{SM}^2+\|(X+\lambda V) u\|_{SM}^2-\left(K_\lambda V u, V u\right)_{S M}\\
& \qquad\qquad -\left(((\kappa_{\lambda})_e+(\eta_{\lambda})_e) V u, V u\right)_{\partial SM},\end{split}
\end{equation}
for all $u\in C^2(SM)$ with $u \circ \rho = u$ on $\partial SM$.
The important point to note here is that the 
regularity of the solution $u$ to the transport equation is $C^2(\operatorname{int}(SM))\cap \operatorname{Lip}(SM)$ by Lemma \ref{lm:reg:u}. We need to prove the Pestov identity \eqref{eq:pestov:1} for this class of functions. To overcome this difficulty, we use an approximation argument following \cite[pp. 1289-1290]{Ilmavirta:Paternain:2022}.
\begin{lemma}\label{lm:pestov:eo}
Let $(M,g)$ be a compact Riemannian surface with smooth boundary and $\lambda \in C^\infty(SM)$. If $u\in C^2(\operatorname{int}(SM))$ $\cap \operatorname{Lip}(SM)$, $Pu\in L^2(SM)$, and $u=u\circ \rho$ on $\partial SM$ then
\begin{align}\label{eq:pestov:eo}
\begin{split}
        \|P u\|_{S M}^2&=\|\widetilde{P} u\|_{S M}^2+\|F u\|_{S M}^2-\left(K_\lambda V u, V u\right)_{S M}\\&\qquad-\left(((\kappa_{\lambda})_e+(\eta_{\lambda})_e) V u, V u\right)_{\partial S M}.
        \end{split}
\end{align}
 \end{lemma}
\begin{proof}
Following the approach taken in the proof of \cite[Lemma 10]{Ilmavirta:Paternain:2022},
we extend our manifold as follows: Let $\widetilde{M}$ be a smooth and compact Riemannian manifold with boundary, such that $M \subset \operatorname{int} \widetilde{M}$. We extend the function $u$ to a new function $\tilde u: S \widetilde{M} \rightarrow \mathbb{R}$ such that $\tilde u$ satisfies $\tilde u=u$ in $SM$, $\tilde u \in C^2(\operatorname{int} S M) \cap \operatorname{Lip}(S \widetilde{M})$ and has compact support in $\operatorname{int} S \widetilde{M}$.

Following the proof of \cite[Lemma 10]{Ilmavirta:Paternain:2022}, we define a sequence of mollifications $\left(u^j\right)_{j=1}^{\infty}$ of $\tilde u$. By the basic properties of mollifiers, we have $u^j \rightarrow u$ in $\operatorname{Lip}(SM)$ and $C^2(\inte SM)$. By applying Lemma \ref{lm:pestov:g} and Lemma \ref{lm:uf:eo:g} to $u^j|_{SM}$, we obtain the following expression
\begin{equation}\label{eq:int:pestov:1}
  \begin{split}
        \|P u^j\|_{SM}^2&=\|\widetilde{P} u^j\|_{SM}^2+\|F u^j\|_{SM}^2-\left(K_\lambda V u^j, V u^j\right)_{S M}\\
&\qquad+\left(\nabla_{T, \lambda} u^j_e, V u^j_o\right)_{\partial S M}+\left(\nabla_{T, \lambda} u^j_o, V u^j_e\right)_{\partial S M}\\
&\qquad\qquad-(((\kappa_{\lambda})_e+(\eta_{\lambda})_e)Vu^j,Vu^j )_{\pi^{-1}  \rR}.
  \end{split}
\end{equation}
Here, we have defined $u_e^j=\frac{1}{2}\left(u^j+u^j \circ \rho\right)$ and $u_o^j=\frac{1}{2}\left(u^j-u^j \circ \rho\right)$, and we have used the fact that $u \circ \rho= u$ at $\partial S M$ by assumption.

Note that $\operatorname{Lip}(\partial SM)\subset H^1(\partial SM)$ and $u^j\to u$ in $\operatorname{Lip}(\partial SM)$. Therefore, the convergence also holds in $H^1(\partial SM)$. Similar to the proof of \cite[Lemma 10]{Ilmavirta:Paternain:2022}, we can conclude from $H^1$ convergence in $SM$ that $Fu^j\to Fu$ and $Vu^j\to Vu$ in $L^2(SM)$. We have by the properties of mollification and regularity assumptions on $u$ that $Pu^j \to Pu$ in $L^2(SM)$. Since the all other terms but $\norm{\widetilde{P}u^j}$ in \eqref{eq:int:pestov:1} are known to converge as $j \to \infty$, we may conclude that $\lim_{j\to \infty}\norm{\widetilde{P}u^j}$ exists and is finite. Using the commutator formula \eqref{eq: twisted-commutators}, we have $\widetilde{P}u^j=Pu^j+X_{\perp}u^j-V(\lambda)Vu^j$. We know that $Pu^j \to Pu$ in $L^2(SM)$ by assumption and, on the other hand, $X_{\perp}u^j \to X_{\perp}u$ and $V(\lambda)Vu^j \to V(\lambda)Vu$ in $L^2(SM)$ since $u \in \mathrm{Lip}(SM) \subset H^1(SM)$ and $\lambda \in C^\infty(SM)$. This shows that $\widetilde{P}u^j \to \widetilde{P}u$ in $L^2(SM)$.

By combining all of the above facts about the convergence of terms, we obtain \eqref{eq:pestov:eo} by taking the limit $j \to \infty$.
\end{proof}
\subsection{Proof of Theorem \ref{thm: main theorem 1}}
Let us write $f=f_{-1}+f_0+f_1$ where $f_j\in H_j$ for $-1\le j\le 1$, as defined in \eqref{eq:GK-eig}. Furthermore, let $(f_{-1}+f_1)(x,v)=\alpha_x(v)$. It follows from the definition of $u$ in \eqref{eq:u:def} that $Fu = -f$ in the interior of $SM$ as stated in \eqref{eq:trans:1}.
By Lemma \ref{lm:reg:u} and the identity $VFu=-Vf \in C^1(SM)$, we know that $u$ satisfies the assumptions of Lemma \ref{lm:pestov:eo}.

From \eqref{eq:GK-eig} and the orthogonality \eqref{eq:orthogonal}, we obtain $VFu = if_1 - if_{-1}$ in $\inte(SM)$ with the identity
\begin{equation}\label{eq:decomposition for proof}
    \|V F u\|_{SM}^2=\|V f\|_{SM}^2=\|if_1-if_{-1}\|_{SM}^2=\|f_1\|_{SM}^2+\|f_{-1}\|_{SM}^2.
\end{equation} 
Combining the Pestov identity \eqref{eq:pestov:eo} with \eqref{eq:decomposition for proof}, we have
\[
\begin{split}
\|f_1\|_{SM}^2+\|f_{-1}\|_{SM}^2&=\|\widetilde{P} u\|_{S M}^2+\|f\|_{SM}^2-\left(K_\lambda V u, V u\right)_{S M} \\
& \qquad-\left(((\kappa_{\lambda})_e+(\eta_{\lambda})_e) V u, V u\right)_{\partial SM}.
\end{split}
\]
We may simplify further to obtain that
\[
\begin{split}
      &0=\|\widetilde{P} u\|_{S M}^2+\|f_0\|_{S M}^2-\left(K_\lambda V u, V u\right)_{S M}\\
      &\qquad -\left(((\kappa_{\lambda})_e+(\eta_{\lambda})_e) V u, V u\right)_{\partial SM}.
\end{split}
\]
Since $(M,g,\lambda,\E)$ is admissible in the sense of Definition \ref{def:admis:lam:brok}, we have that $(\kappa_{\lambda_e} + \eta_{\lambda_e}) \le 0$ and $K_{\lambda} \le 0$. In other words, each term on the right-hand side of the above equation is individually nonpositive. Since the sum of these terms vanishes, it follows that each individual term must be zero. Consequently, we deduce that $FVu = 0$ and $f_0=0$. Since the dynamical system is nontrapping and $Vu$ is constant along the $\lambda$-geodesic flow, this implies that $Vu=0$ by the boundary condition $u|_{\pi^{-1}\E}=0$. In conclusion, $u$ is independent of the vertical variable $v$. Consequently, we have
\[f=f_1+f_{-1}=F(-u)=(X+\lambda V)(-u)=-Xu.
\]
Furthermore, $u=-\pi^*h$ for some $h:M\to \mathbb R$ satisfying $dh=\alpha$. 
Since $\alpha$ is an exact $1$-form with $C^2$ coefficients and $h$ itself is continuous, we may conclude that $h \in C^3(M)$.\qed
\appendix

\section{Geometry of twisted geodesic flows on surfaces}
In this appendix, for the sake of completeness, we develop some basic theory for the twisted geodesic flows. Most of these properties and lemmas are used in the article. This includes the Pestov energy identities with boundary terms, some properties of $\lambda$-Jacobi fields, the time-reversed flows (called the dual flows), and a discussion on different notions of curvature and convexity. We remark that our discussion complements those of \cite{Assylbekov:Dairbekov:2018,Dairbekov:Paternain:2007:Entropy1,Zhang:2023}. In particular, our Proposition \ref{lm:pestov:g} can be seen as a special case of \cite[Theorem 2.3]{Assylbekov:Dairbekov:2018}, using a slightly different approach and notation. We also remark that $\lambda$-Jacobi fields on surfaces were already introduced and studied in the context of $\lambda$-conjugate points in \cite{Assylbekov:Dairbekov:2018,Zhang:2023}. However, Lemma \ref{lm:lambda:jac:var}, the implications of it and the growth estimate of Lemma \ref{lm:Jacobi:est} do not appear in these works. On the other hand, the dual flow, in Section \ref{subsec:dualflows}, is only introduced and applied in our work.

\subsection{Proof of Proposition \ref{lm:pestov:g}} \label{sec:proof-of-pestov}
We first assume that $u \in C^{\infty}(SM)$. Then the proof of proposition for $C^2(SM)$ follows by the density of $C^{\infty}(SM)$ in $C^2(SM)$. Using the commutator formulas \eqref{eq: twisted-commutators}, for any smooth function $u:SM \rightarrow\mathbb{R}$ one has 
\begin{equation}\label{eq:int:1}
\begin{split}
    & -2(X_{\perp}u\cdot V(X+\lambda V)u) \\
     &\qquad \qquad= ((X+\lambda V)u)^2+(X_{\perp}u)^2-(K+X_{\perp}(\lambda)+\lambda^2 )(Vu)^2\\
   &\qquad\qquad\qquad -(X+\lambda V+V(\lambda))(X_{\perp}u\cdot Vu )+X_{\perp}((X+\lambda V)u\cdot Vu) \\
   &\qquad\qquad\qquad\qquad-V((X+\lambda V)u\cdot X_{\perp}u
   ).
\end{split}
\end{equation}
See \cite[Lemma 3.1]{Dairbekov:Paternain:2007:Entropy1} for the details.
Integrating the identity \eqref{eq:int:1} over $SM$, we obtain
\begin{align*}
   &-2\int_{SM}(X_{\perp}u\cdot V(X+\lambda V)u) d \Sigma^3 \\
   &\quad\qquad=\int_{SM}((X+\lambda V)u)^2 d \Sigma^3+ \int_{SM}(X_{\perp}u)^2d \Sigma^3\\
    &\qquad\qquad - \int_{SM}(K+X_{\perp}(\lambda)+\lambda^2 )(Vu)^2d \Sigma^3-
    \int_{SM}  (X+\lambda V+V(\lambda))(X_{\perp}u\cdot Vu )d \Sigma^3\\
    &\qquad\qquad \qquad+\int_{SM}X_{\perp}((X+\lambda V)u\cdot Vu) d \Sigma^3- \int_{SM}V((X+\lambda V)u\cdot X_{\perp}u 
   )d \Sigma^3.
\end{align*} 
Using the integration by parts formulas \eqref{eq:int:parts}, we have
\begin{align*}
    \int_{SM}  (X+\lambda V+V(\lambda))(X_{\perp}u\cdot Vu )d \Sigma^3&=- \int_{\partial SM}\langle v,\nu\rangle(X_{\perp}u\cdot Vu )d \Sigma^2,\\
    \int_{SM}(X_{\perp})((X+\lambda V)u\cdot Vu) d \Sigma^3&=-\int_{\partial SM} \langle v_{\perp},\nu\rangle ((X+\lambda V)u\cdot Vu)d \Sigma^2,\\
    \int_{SM}V((X+\lambda V)u\cdot X_{\perp}u 
   )d \Sigma^3&=0.
\end{align*}
This implies
\begin{align}\label{eq:int:2}
\begin{split}
     &-2\int_{SM}(X_{\perp}u)\cdot V(X+\lambda V)u d \Sigma^3\\
     &\qquad=\int_{SM}((X+\lambda V)u)^2 d \Sigma^3+ \int_{SM}(X_{\perp}u)^2d \Sigma^3 - \int_{SM}(K+X_{\perp}(\lambda)+\lambda^2 )(Vu)^2d \Sigma^3\\
     &\qquad\qquad+\int_{\partial SM}\langle v,\nu\rangle(X_{\perp}u\cdot Vu )d \Sigma^2 -\int_{\partial SM} \langle v_{\perp},\nu\rangle ((X+\lambda V)u\cdot Vu)d \Sigma^2.
    \end{split}
\end{align}
Applying the integrating by parts formula once again, we have
\begin{align}\label{eq:int:12}
   \begin{split}
       & \int_{SM} (X+\lambda V)(V(\lambda)(Vu)^2)d \Sigma^3\\
       &\qquad=-\int_{SM}(V(\lambda))^2(Vu)^2d \Sigma^3-\int_{\partial SM} \langle v,\nu\rangle (V(\lambda))(Vu)^2 d \Sigma^2.
   \end{split}
\end{align}
An identity similar to \cite[p. 538]{Dairbekov:Paternain:2007:Entropy1}
can be obtained using the commutator relations (as illustrated in \eqref{eq: twisted-commutators}), as follows
\[
\begin{split}
    ((X+\lambda V)Vu)^2
    &=(V(X+\lambda V)u)^2+(X_{\perp}u)^2-(V(\lambda ))^2(Vu)^2
    +2V(X+\lambda V)u\cdot X_{\perp}u \\
     &\qquad -(X+\lambda V)((V(\lambda))(Vu)^2)+(Vu)^2(X+\lambda V)(V(\lambda)).
\end{split}
\]
We now integrate the above equation over $SM$ to get
\begin{align}\label{eq:int:3}
    \begin{split}
      & -2\int_{SM} X_{\perp}u\cdot V(X+\lambda V)u d \Sigma^3\\
       &\qquad \qquad= -\int_{SM}((X+\lambda V)Vu)^2  d \Sigma^3 
    + \int_{SM}  (V(X+\lambda V)u)^2  d \Sigma^3
       \quad + \int_{SM}(X_{\perp}u)^2 d \Sigma^3  \\  & \qquad\qquad\qquad+ \int_{SM} ((X+\lambda V)(V(\lambda))(Vu)^2d \Sigma^3
     +\int_{\partial SM} \langle v,\nu\rangle (V(\lambda))(Vu)^2 d \Sigma^2.
    \end{split}
\end{align}
Notice that the identity above is a generalization of \cite[eq.(14)]{Dairbekov:Paternain:2007:Entropy1} with the boundary term. Combining \eqref{eq:int:2} with \eqref{eq:int:3} yields 
\[
\begin{split}
   & \int_{SM}((X+\lambda V)Vu)^2  d \Sigma^3  - \int_{SM}(K+X_{\perp}(\lambda)+\lambda^2 +(X+\lambda V)V(\lambda))(Vu)^2d \Sigma^3\\
    &\qquad= \int_{SM}  (V(X+\lambda V)u)^2  d \Sigma^3
     -\int_{SM}((X+\lambda V)u)^2 d \Sigma^3 +\int_{\partial SM} \langle v,\nu\rangle (V(\lambda))(Vu)^2 d \Sigma^2\\
     &\qquad\qquad+\int_{\partial SM} \langle v_{\perp},\nu\rangle ((X+\lambda V)u\cdot Vu)d \Sigma^2-\int_{\partial SM}\langle v,\nu\rangle(X_{\perp}u\cdot Vu )d \Sigma^2.
\end{split}
\]
Hence, we have
\[
\begin{split}
    \|Pu\|_{SM}^2&= \|\widetilde{P}u \|_{SM}^2+ \|Fu\|_{SM}^2 -(K_{\lambda}Vu,Vu )_{SM}
  -( \langle v_{\perp},\nu\rangle Fu, Vu)_{\partial SM} \\ &\quad-(\langle v,\nu\rangle (V(\lambda))Vu,Vu )_{\partial SM} +( \langle v,\nu\rangle(X_{\perp}u, Vu ))_{\partial SM}.
\end{split}
\]
This is precisely the assertion of the proposition.
\qed
\subsection{Jacobi fields for \texorpdfstring{$\lambda$}{lambda}-geodesics}\label{sec: lambda-jacobi-fields} Jacobi fields are understood by means of a variation of a family of geodesics (see e.g. \cite[Chapter 10]{Lee:riem:2nd} or \cite[Section 3.7.1]{GIP2D}). In this subsection, we provide a detailed exposition of the Jacobi field for $\lambda$-geodesic flows.
Consider a smooth one-parameter family of $\lambda$-geodesics $\left(\gamma_s\right)_{s \in(-\varepsilon, \varepsilon)}$.
 We say that the family $\left(\gamma_s\right)_{s \in(-\varepsilon, \varepsilon)}$ is a variation of $\gamma$ through $\lambda$-geodesics if each $\gamma_s:[a, b] \rightarrow M$ is a $\lambda$-geodesic and $\gamma_0=\gamma$. We denote by $D_t = \nabla_{\dot\gamma(t)}$ the covariant derivative along the curve $\gamma(t)$.
 
It is well-known that the variation field of a geodesic satisfies the Jacobi equation which is demonstrated in \cite[Lemma 3.7.2]{GIP2D}, \cite[Theorem 10.1]{Lee:riem:2nd}. Furthermore, the magnetic Jacobi fields are known to satisfy the equation \eqref{eq:jac:var:lambda} for $\lambda \in C^{\infty}(M)$ (cf. \cite[eq. (A.7)]{Dairbekov-Paternain-SU-2007-magnetic-rigidity}). In our next lemma, we deduce the $\lambda$-Jacobi equation in the case of $\lambda\in C^{\infty}(SM)$.
\begin{lemma}[$\lambda$-Jacobi equation]\label{lm:lambda:jac:var}Let $(M, g)$ be a Riemannian surface with or without boundary and $\lambda \in C^{\infty}(SM)$. Let $\gamma$ be a $\lambda$-geodesic segment in $M$ and $\gamma_s$ be a variation of $\gamma$. Then the variation field $J(t) = \left.\partial_s \gamma_s(t)\right|_{s=0}$ of a variation through $\lambda$-geodesics satisfies the $\lambda$-Jacobi equation
\begin{equation}\label{eq:jac:var:lambda}
     D_t^2 J(t)=  \langle(J(t), \dot{J}(t)), \nabla_{S M} \lambda\rangle_G \ i \dot{\gamma}(t)+ \lambda(\gamma(t), \dot{\gamma}(t)) \nabla_{J}(i\dot\gamma)+R(\dot{\gamma}(t), J(t)) \dot{\gamma}(t)
\end{equation}
where $R$ is the Riemann curvature tensor of $(M,g)$ and $G$ is the Sasaki metric on $SM$.
\end{lemma}
\begin{remark}
In the article, we will only consider variations of a unit speed $\lambda$-geodesic. Therefore, Jacobi fields have only one tangential component as the scaling of speed is not permissible under this assumption.
\end{remark}
\begin{proof} 
Let $\Gamma(s, t) = \gamma_s(t)$ be a smooth variation of $\lambda$-geodesics on the Riemannian surface $M$. Note that the family  $\gamma_s(t)$ of $\lambda$-geodesics satisfies
\begin{equation}\label{eq:geodesic:var}
    D_t \partial_t \gamma_s(t)=\lambda i \dot{\gamma}_s.
\end{equation} Define $J(t) = \left.\partial_s \gamma_s(t)\right|_{s=0}$. We first compute $D_t^2 J(t)$.
We denote by $D_s = \nabla_{\partial_s\gamma_s}$ the covariant derivative along $\partial_s \gamma_s$. 
According to \cite[Lemma 6.2]{Lee:riem:2nd}, we have the torsion-free property of the connection $\nabla$,  
\begin{equation}\label{eq:symmetry:thm}
    D_t \partial_s \gamma_s(t)=D_s \partial_t \gamma_s(t).
\end{equation}
Recall that the Riemann curvature tensor $R$ satisfies 
\begin{equation}\label{eq:curv:jac}
    D_t D_s W-D_s D_t W=R\left(\partial_t \gamma_s, \partial_s \gamma_s\right) W
\end{equation}
for any vector field $W$ along $\gamma_s$, see \cite[Proposition 7.5]{Lee:riem:2nd}.
\begin{comment}
\end{comment}
Combining \eqref{eq:symmetry:thm}, \eqref{eq:curv:jac} and \eqref{eq:geodesic:var}, we have
\begin{equation}\label{eq:j:int::1}
    \begin{split}
     D_t^2 J(t) & =\left.D_t D_t \partial_s \gamma_s(t)\right|_{s=0}  =\left.D_t D_s \partial_t \gamma_s(t)\right|_{s=0}
     \\ & =\left.D_s D_t \partial_t \gamma_s(t)\right|_{s=0}+R(\dot{\gamma}(t), J(t)) \dot{\gamma}(t)\\
&=D_s (\lambda(\gamma_s(t),\dot\gamma_s(t) )i\dot\gamma_s(t))|_{s=0}+R(\dot{\gamma}(t), J(t)) \dot{\gamma}(t).
\end{split}
\end{equation}
Applying the product rule for the covariant derivative along a curve for a scalar function $f$ and a vector field $V$, we have
\[
    D_s(f V)=\left(\partial_s f\right) V+f D_s V,
\]
see for instance,  \cite[Theorem 4.24]{Lee:riem:2nd}. This implies  
\begin{equation}\label{eq:j:int:2}
    D_s (\lambda(\gamma_s(t),\dot\gamma_s(t) )i\dot\gamma_s(t))=\partial_s(\lambda(\gamma_s(t),\dot\gamma_s(t))) i\dot\gamma_s(t)+\lambda(\gamma_s(t),\dot\gamma_s(t)) D_s(i\dot{\gamma_s}(t)).
\end{equation}
Since $\lambda\in C^{\infty}(SM)$, we see that
\begin{align}\label{eq:j:int:3}
    \partial_s\left(\lambda\left(\gamma_s(t), \dot{\gamma}_s(t)\right)\right)|_{s=0}&=\langle ( \partial_s \gamma_s(t), D_s\dot\gamma_s(t) )|_{s=0}, \nabla_{SM}\lambda \left(\gamma_s(t), \dot{\gamma}_s(t)\right)|_{s=0}\rangle_{G}\notag\\
    &= \langle (J(t),D_tJ(t) ), \nabla_{SM}\lambda (\gamma(t),\dot\gamma(t)) \rangle_G.
\end{align}
Combining \eqref{eq:j:int::1}, \eqref{eq:j:int:2} and \eqref{eq:j:int:3}, we have
\begin{align*}
    D_t^2 J(t)&=\left\langle\left(J(t), D_t J(t)\right), \nabla_{S M} \lambda(\gamma(t), \dot{\gamma}(t))\right\rangle_G i\dot \gamma(t)+\lambda(\gamma(t), \dot{\gamma}(t)) \nabla_J(i \dot{\gamma})\\&\qquad+R(\dot{\gamma}(t), J(t)) \dot{\gamma}(t).
\end{align*}
We thus have proved the lemma.
\end{proof}
Recall from \cite[p. 82]{GIP2D}, the connection map $\mathfrak{K} : T_{(x,v)}SM \to T_xM$ is defined as follows: For a given $\xi \in T_{(x,v)}SM$, consider a curve $Z : (-\epsilon, \epsilon) \to SM$ with $Z(0) = (x,v)$ and $\dot{Z}(0) = \xi$. We can represent the curve $Z(s)$ as a pair $(\alpha(s), W(s))$. Then the connection map acts on $\xi$ as $\mathfrak{K}\xi := \left. D_sW \right|_{s=0}$, where $D_s$ denotes the covariant derivative of $W$ along $\alpha$. From \cite[eq. (3.12)]{GIP2D}, any $\xi\in T_{(x,v)}SM$ can be written as \begin{align}\label{eq:spliting}
    \xi=(d\pi(\xi),\mathfrak K\xi)
\end{align}
where $d \pi(\xi)$ lies in the horizontal subbundle and $\mathfrak K\xi$ resides in the vertical subbundle.
\begin{remark}\label{rm:decompositon}
    To simplify, we need to decompose $(J(t),D_tJ(t) )$ into horizontal and vertical sub bundles in order to use the Sasaki metric property. Now, 

\begin{align*}
    J(t) &= \left.\partial_s \gamma_s(t)\right|_{s=0}=\left.\partial_s \pi(\phi_t(\gamma_s(0))\right|_{s=0}= d\pi(\phi_t (\gamma_s(0)))\partial_s\phi_t(\gamma_s(0))|_{s=0}
\end{align*}
and 
\begin{align}\label{eq:int:j:100}
    D_tJ(t)&=\left.D_t\partial_s \gamma_s(t)\right|_{s=0}=\left.D_s\partial_t \gamma_s(t)\right|_{s=0}=\left.D_s\dot\gamma_{s}(t)\right|_{s=0}.
\end{align}
Consider the map $Z_t(s)=\phi_t(\gamma_s(0))=(\gamma_s(t),\dot\gamma_s(t) )$. By the definition of the connection map $\mathfrak K$, we have that $\mathfrak K (\partial_s \phi_t(\gamma_s(0))|_{s=0} ) =D_s\dot\gamma_{s}(t)|_{s=0}$. Thus, we can deduce from \eqref{eq:int:j:100} that $D_tJ(t)=\mathfrak K (\partial_s \phi_t(\gamma_s(0))|_{s=0} )$.
According to \cite[eq. (3.12)]{GIP2D}, any vector $\xi \in T_{(x, v)} S M$ can be decomposed as $\xi=\left(\xi_H, \xi_V\right)$, where $\xi_H=d \pi(\xi)$ represents the horizontal component and $\xi_V=\mathfrak K \xi$ denotes the vertical component.
Now by taking $\xi=\partial_s \phi_t(\gamma_s(0))|_{s=0}$, we can write 
\begin{align*}
   \partial_s \phi_t(\gamma_s(0)= (J(t),D_tJ(t) )
\end{align*}
where $J(t)=(  \partial_s \phi_t(\gamma_s(0))|_{s=0})_H$ and $D_tJ(t)=(  \partial_s \phi_t(\gamma_s(0))|_{s=0})_V$. Additionally, the gradient of $\lambda$ on $SM$, denoted by $\nabla_{SM}\lambda$, can be decomposed into its horizontal and vertical components, expressed as $((\nabla_{SM}\lambda)_H,(\nabla_{SM}\lambda)_V)$. By the definition of the Sasaki metric (see for instance \cite[eq. (3.14)]{GIP2D}), we have
\begin{align}
    \left\langle(J(t), \dot{J}(t)), \nabla_{S M}(\lambda(\gamma(t), \dot{\gamma}(t)))\rangle_G=\left\langle J(t),\left(\nabla_{S M} \lambda\right)_H\right\rangle_g+\left\langle D_t J(t),\left(\nabla_{S M} \lambda\right)_V\right\rangle_g\right.. 
\end{align}
\end{remark}

\begin{definition}[$\lambda$-exponential map]
 Let $(M,g)$ be a Riemannian surface with or without boundary and $\lambda\in C^{\infty}(SM)$. For each point $x \in M$, let us define the maximal domain $D_x$ as follows:
\[D_x:=\left\{t v \in T_x M ; v \in S_x M \text { and } t \in[0, \tau(x, v)]\right\}.\]
The \emph{$\lambda$-exponential map} $\exp_x^{\lambda}: D_x \rightarrow M$ is given by $\exp_x^{\lambda}(t v) = \gamma_{x, v}(t)$, where $\gamma$ is a $\lambda$-geodesic.
\end{definition}
\begin{remark}
We set $\tau(x,v):=\infty$ if $\partial M = \emptyset$ and, in general, if the flow $\phi_t(x,v)$ does not reach the boundary.
\end{remark}
\begin{remark}
Notice that $\exp _x^\lambda(tv)=\pi\circ \phi_t(x,v)$ for $t\ge 0$, where $\pi$ is the projection $\pi(x, v) = x$. Hence, the $\lambda$-exponential map is smooth on $D_x \setminus {0}$.
\end{remark}
\noindent We denote $\lambda(t):=\lambda(\gamma(t),\dot\gamma(t) )$. The following lemma is a generalization of \cite[Proposition 10.2]{Lee:riem:2nd} in the case of $\lambda$-Jacobi field.
\begin{lemma}[Existence and uniqueness of $\lambda$-Jacobi fields]\label{lemma:existence and uniqueness of Jacobi fields} 
Let $(M, g)$ be a Riemannian surface with or without boundary and $\lambda \in C^{\infty}(SM)$. Suppose $I\subset\mathbb{R}$ is an interval and $\gamma: I \rightarrow M$ is a $\lambda$-geodesic with $\gamma(t_0) = x$ for some $t_0 \in I$. For any pair $(v, w) \in S_xM \times T_xM$, there exists a unique $\lambda$-Jacobi field $J$ along $\gamma$ satisfying the initial conditions 
\[
J(t_0) = v \ \text{and}\ \ D_tJ(t_0) = w.
\]
\end{lemma}

The following lemma is a converse of Lemma \ref{lm:lambda:jac:var}, which tells that any $\lambda$-Jacobi field is the variation field of a variation of $\lambda$-geodesics. A similar result has been proved in \cite[Proposition 10.4]{Lee:riem:2nd} in the case of $\lambda=0$.
\begin{lemma}\label{lm:converse:jacobi}
Let $(M, g)$ be a Riemannian surface with or without boundary and $\lambda \in C^{\infty}(SM)$.
 Suppose $I\subset\mathbb{R}$ is a compact interval and $\gamma: I \rightarrow M$ is a $\lambda$-geodesic. Then every $\lambda$-Jacobi field along $\gamma$ corresponds to the variation field of a unit speed $\lambda$-geodesic variation of $\gamma$.
\end{lemma}
\noindent 
We omit the proofs of Lemmas \ref{lemma:existence and uniqueness of Jacobi fields} and \ref{lm:converse:jacobi}. One may prove them following the standard Riemannian proofs with obvious modifications.

We now introduce the notion of Lie bracket, following \cite[Remark 4.13]{Marry:Paternain:2011:notes} and we require this in our later analysis. Let $Y$ and $Z$ be two vector fields on a manifold $N$. We will denote by $\psi_t$ the local flow of the vector field $Y$.
The Lie bracket of $Y$ and $Z$, denoted by $[Y, Z]$, can then be defined by setting 
\begin{equation}\label{eq:lie:deri}
    [Y, Z](x)=\left.\frac{d}{d t}\right|_{t=0} d \psi_{-t}\left(Z\left(\psi_t x\right)\right).
\end{equation}
Let $(x,v) \in SM$, $\xi \in T_{(x,v)}SM$ and $\phi_t$ be a $\lambda$-geodesic flow.
Similarly to \cite[p. 38]{Marry:Paternain:2011:notes}, we write $F(t):=F(\phi_t(x,v))$, $X_{\perp}(t):=X_\perp(\phi_t(x,v))$ and $V(t) := V(\phi_t(x,v))$ with a slight abuse of notation. We can express $\xi$ as $\xi=aF+bX_{\perp}+cV$, where $a,b,c\in \mathbb R$ and ${F,V,X_{\perp}}$ form a basis for $T_{(x,v)}SM$, see in \cite[eq. (4.2.2.)]{Marry:Paternain:2011:notes}. Furthermore, we can find smooth functions $a(t), b(t), c(t)$ satisfying 
\begin{equation}\label{eq:in:a0}
d \phi_t(\xi)=a(t) F(t)+b(t) X_{\perp}(t)+c(t) V(t),
\end{equation}
with initial conditions $a(0)=a$, $b(0)=b$, and $c(0)=c$. We assume here that \eqref{eq:in:a0} is defined for all $t \in I $, where the interval $I$ might be unbounded. When $M$ is a closed surface, then one may take $I = \R$ as the $\lambda$-geodesic flow is defined for all times.

The coefficients $a(t),b(t),c(t)$ satisfy a certain system of ODEs (see e.g. \cite[Proposition 4.14]{Marry:Paternain:2011:notes}). In the next lemma, we will prove a similar result in the case of $\lambda$-geodesics corresponding to the $\lambda$-Jacobi equation. See also \cite{Assylbekov:Dairbekov:2018,Mettler-Paternain-vortices-2022,Zhang:2023} for an equivalent result. In comparison to \cite[Section 3.2]{Mettler-Paternain-vortices-2022}, we remark that our definition of $K_{\lambda}$ is different by including the term $FV(\lambda)$ in $K_{\lambda}$, $H=-X_\bot$, and the formulas look a bit different because of these conventions.
\begin{lemma}
Let $(M, g)$ be a closed surface and $\lambda \in C^{\infty}(SM)$. If functions $a(t), b(t),$ and $c(t)$ satisfy the equation \eqref{eq:in:a0} with the initial conditions $a(0) = a$, $b(0) = b$, and $c(0) = c$, then these functions satisfy the following set of differential equations
\begin{align}\label{eq:in:a1}
\begin{split}
    \dot{a} & =-\lambda b, \\
\dot{b} & =-c, \\
\dot{c}-cV(\lambda  )& =(K_{\lambda}-F V(\lambda))b  .
\end{split}
\end{align}
\end{lemma}
\begin{proof} Define $\phi_{-t}:=r\circ \phi_t^{-}\circ r$, where $\phi_t^{-}$ is the dual $\lambda$-geodesic flow and $r$ is reversion map i.e., $r(x,v)=(x,-v)$. Since composition of smooth maps is smooth and $\phi_{-t}\circ\phi_t=\operatorname{Id} $ in $SM$, we have $\phi_{-t}$ is a smooth map. Therefore, for each $t \in [0, \infty)$, $\phi_t: SM \to SM$ is a diffeomorphism. This implies that the differential $d\phi_t$ is invertible for every $t \in [0, \infty)$. 
 Consequently, the differential $d\phi_{-t}$ is invertible for all $t \in [0, \infty)$ and hence $d \phi_t^{-1}=d \phi_{-t}$.
   Similar to \cite[Proof of Proposition 4.14]{Marry:Paternain:2011:notes}, we apply $d\phi_{-t}$ to both sides of \eqref{eq:in:a0}. Then we obtain
   \begin{equation}\label{eq:in:a3}
       \xi=a(t) d \phi_{-t}(F(t))+b(t) d \phi_{-t}(X_{\perp}(t))+c(t) d \phi_{-t}(V(t)).
   \end{equation}
   Differentiating both sides of the above identity with respect to $t$ and applying the derivative formula for Lie brackets \eqref{eq:lie:deri}, we have 
   \begin{align*}
0=  \frac{d}{d t}(\xi) 
= & \dot{a}(t) d \phi_{-t}(F(t))+a(t) d \phi_{-t}([F, F](t))+\dot{b}(t) d \phi_{-t}(X_{\perp}(t)) \\
& +b(t) d \phi_{-t}([F, X_{\perp}](t))+\dot{c}(t) d \phi_{-t}(V(t))+c(t) d \phi_{-t}([F, V](t)).
\end{align*}
Applying the commutator formulas, linearity and grouping like terms, we obtain
\begin{align*}
    0&=d \phi_{-t}\{(\dot{a}(t)+\lambda b(t)) F(t)+(\dot{b}(t)+c(t)) X_{\perp}(t)\\
    &\qquad+(\dot{c}(t)-(K_{\lambda}-F V(\lambda)) b(t)-c(t)V(\lambda)) V(t)\}.
\end{align*}
Since $d\phi_{-t}$ is invertible and $\{F(t), X_{\perp}(t), V(t)\}$ is a basis of each tangent space $T_{\phi_t\left(x, v\right)} S M$, the coefficients of $F(t), X_{\perp}(t)$ and $V(t)$ must vanish for all $t$, and hence the result follows.
\end{proof}

\begin{lemma}\label{lm:differential_flow}
Let $(M,g)$ be a  Riemannian surface with or without boundary and $\lambda \in C^{\infty}(SM)$.  For any $(x, v) \in SM$ with $\xi \in T_{(x, v)} SM$, the differential of the $\lambda$-geodesic flow can be decomposed as 
\begin{equation}\label{eq:in:a4}
    d \phi_t(\xi)=\left(J_{\xi}(t), D_t J_{\xi}(t)\right),
\end{equation} 
where $J_{\xi}$ is the unique Jacobi field along the $\lambda$-geodesic $\pi(\phi_t(x, v))$ with the initial condition $(J_{\xi}(0), D_t J_{\xi}(0))=\xi$.
\end{lemma}
\begin{proof}
Let us consider $\sigma:(-\epsilon,\epsilon)\to SM$ be a smooth curve such that $\sigma(0)=(x,v)$ and $\sigma'(0)=\xi$. Now we consider a variation of the $\lambda$-geodesic 
\begin{align*}
    \Gamma(t,s):=\gamma_{\sigma(s)}(t)=\pi(\phi_t(\sigma(s))).
\end{align*}
Now we have 
\begin{align*}
J(t)&:=\partial_s\Gamma(t,s)|_{s=0}=\partial_s\pi(\phi_t(\sigma(s))|_{s=0}=d\pi(\phi_t(\sigma(0)))d\phi_t(\sigma(0))\sigma'(0)\\
    &=d\pi(\phi_t(x,v))d\phi_t(x,v)\xi.
\end{align*}
Let $Z_t(s):=\phi_{t}(\sigma(s))=(\gamma_{\sigma(s)}(t),\dot\gamma_{\sigma(s)}(t))$. Using the definition of connection map, we have\begin{equation}\label{eq:int:ap1}
    \mathfrak K (\partial_s Z_t(s)|_{s=0} )=\mathfrak K (\partial_s\phi_{t}(\sigma(s))|_{s=0} )=D_s\dot\gamma_{\sigma(s)}(t)|_{s=0}.
\end{equation}
Combining \eqref{eq:symmetry:thm} with \eqref{eq:int:ap1}, we deduce that
\begin{align*}
    D_tJ(t)&=D_t\partial_s\Gamma(t,s)|_{s=0}=D_s\partial_t\Gamma(t,s)|_{s=0}\\
    &=D_s\partial_t\pi(\phi_t(\sigma(s)))|_{s=0}=D_s\dot\gamma_{\sigma(s)}(t)|_{s=0}\\&=\mathfrak K\left( \partial_s(\phi_t(\sigma(s)))|_{s=0}\right)=\mathfrak K d\phi_t(\sigma(0))\sigma'(0)\\&=\mathfrak K d\phi_t(x,v)\xi.
\end{align*}
From \eqref{eq:spliting}, we obtain
\begin{align*}
    d \phi_t(x, v) \xi&=(d\pi (\phi_t(x,v)) d \phi_t(x, v) \xi,\mathfrak K  d \phi_t(x, v) \xi )\\
    &=(J(t), D_tJ(t) ).
\end{align*}
This result establishes the lemma. 
\end{proof}

We will next proof an estimate needed for the proof of our main theorem.
\begin{proof}[Proof of Lemma \ref{lm:Jacobi:est}]
If $M$ is a surface with boundary, then we extend $(M,g)$ into a closed surface $N$ and extend $\lambda$ to a smooth function on $SN$ (cf. \cite[Lemma 3.1.8]{GIP2D}). Then the $\lambda$-geodesic flow is defined for all $t\in  \mathbb R$. If \eqref{eq:jac_est} holds for the closed extension $N$, then it also holds for $M$. Therefore we can assume without loss of generality, we are in a closed setting.

   Let us first define $Z = (Z_1, Z_2,Z_3)$ such that $Z_1 := a, Z_2 := b, Z_3:=c$. Then the equation \eqref{eq:in:a1} can be written as 
\begin{equation}\label{JacoEq}
\begin{cases}
D_t Z_1 = \dot{a}=-\lambda(\gamma,\dot \gamma)Z_2, \\
D_t Z_2=- Z_3,
\\
D_t Z_3= V(\lambda)Z_3+(K_{\lambda}-F V(\lambda))Z_2.
\end{cases}
\end{equation}
This is equivalent to
\[
D_t Z = A_{t}^{\gamma}Z,
\]
where
\[
A_{t}^{\gamma} =
\left[ {\begin{array}{ccc}
0&-\lambda &0\\0&0&-1 \\ 
0&(K_{\lambda}-F V(\lambda))&V(\lambda)
\end{array} } \right]
\]
is a bounded linear map for each $t$. By compactness, there is a positive constant $C=C(M,g,\lambda)$ such that $\left\|A_{t}^{\gamma}\right\| \leq C / 2$ for all geodesics $\gamma$ and all times $t$.
Thus, we have
$$
D_{t}|Z|^{2}=2\left\langle Z, A_{t}^{\gamma} Z\right\rangle \leq C|Z|^{2} .
$$
By Gr\"onwall's inequality (cf. \cite[p. 624]{Evans:book:PDE}), the above implies that $|Z(t)|^{2} \leq e^{C t}|Z(0)|^{2}$ for all $t \geq 0$.

Now if $J$ is a $\lambda$-Jacobi field, then from \eqref{eq:in:a0} and \eqref{eq:in:a4}, we have $Z=\left(J, D_{t} J\right)$ for some initial data, from which the claim follows.
\end{proof}
\subsection{Convexity, concavity and signed \texorpdfstring{$\lambda$}{lambda}-curvature}\label{sec:signed-preliminaires}
Recall that $\sff$ denotes the second fundamental form of the boundary $\partial M$, and $\nu(x)$ the inward unit normal vector to $\partial M$ at $x\in \partial M$.

\begin{definition}\label{def:lam_convex} Let $(M, g)$ be a Riemannian surface with smooth boundary and $\lambda \in C^{\infty}(SM)$. We say that the boundary $\partial M$ is \emph{strictly $\lambda$-convex at a point $x \in \partial M$} if the following inequality holds: 
$$
\sff_x(v, v) >\left\langle \lambda(x,v)iv, \nu(x)\right\rangle \quad \text { for all } v \in S_x(\partial M) .
$$ 
Similarly, we say that the boundary $\partial M$ is \emph{strictly $\lambda$-concave at a point $x \in \partial M$} if the following inequality holds:
$$
\sff_x(v, v) <\left\langle \lambda(x,v)iv, \nu(x)\right\rangle \quad \text { for all } v \in S_x(\partial M) .
$$
\end{definition}
In \cite[Lemma 3.1.12]{GIP2D}, it is demonstrated that any geodesic tangent to the boundary $\partial M$ stays in the exterior to the surface $M$ for both small positive and negative time intervals. Moreover, the surface $M$ is strictly magnetic convex at $x\in \partial M$, if $\sff(x, v) >\left\langle Y_x(v), \nu(x)\right\rangle$, holds for all $v \in S_x(\partial M)$ where $Y_x$ denotes the Lorentz force, see \cite[Lemma A.6]{Dairbekov-Paternain-SU-2007-magnetic-rigidity}. An analogous property will be discussed in the next lemma in the case of a strictly $\lambda$-convex boundary which generalizes both of these results. Before stating the next lemma, let us consider the boundary defining map following
\cite[Lemma 3.1.10]{GIP2D}. Let 
$(N,g)$ be a closed extension of $(M,g)$. Then there is a function $\rho\in C^{\infty}(N)$, called a boundary defining function, such that $\rho(x)=d(x,\partial M)$ near $\partial M$ in $M$, and $M = \{x\in N : \rho\geq 0\}, \partial M = \{x\in N : \rho= 0\},$ and $N\setminus M = \{x\in N : \rho< 0\}.$ Moreover, $\nabla \rho(x) = \nu(x)$ holds for all $x\in \partial M$.
\begin{lemma}
Let $(M,g)$ be a compact Riemannian surface with smooth boundary and $\lambda \in C^{\infty}(SM)$. Suppose $(N,g)$ is a closed extension of $(M,g)$. Then $\partial M$ is strictly $\lambda$-convex if and only if any $\lambda$-geodesic in $N$ starting from some point $(x,v)\in \partial_0SM$ satisfies $ \frac{d^2}{d t^2}[\rho \circ \gamma_{x,v}(t)]|_{t=0} <0.$ Furthermore, any $\lambda$-geodesic tangent to $\partial M$ stays outside $M$ for small positive and negative time intervals. Also any maximal $\lambda$-geodesic going from 
$\partial M$ into $M$ stays in the interior of $M$ excepts for its endpoints.
\end{lemma}
\begin{proof}
Let $\rho$ be a boundary defining function such that $\rho|_{\partial M}=0$, $\nabla \rho|_{\partial M}=\nu$, $\rho|_{M}\ge 0$ and $\rho|_{N\setminus M}<0$, see for instance \cite[Lemma 3.1.10]{GIP2D}. Let $v \in S_x(\partial M)$ and $\gamma_{x,v}(t)$ be the $\lambda$-geodesic with $\gamma_{x,v}(0)=x, \dot{\gamma}_{x,v}(0)=v$. Now $\partial M$ is strictly $\lambda$-convex if and only if 
$\sff(x,v)>\langle \nu, \lambda(x,v)iv \rangle.$ Notice that
\begin{align*}
  \left.  \frac{d^2}{d t^2}[\rho \circ \gamma_{x,v}(t)]\right|_{t=0}&= \left. \frac{d}{d t}\langle\nabla \rho(\gamma_{x,v}(t)), \dot{\gamma}(t)\rangle\right|_{t=0} \\ & =\left\langle\nabla_{\dot{\gamma}(t)} \nabla \rho(\gamma_{x,v}(t)), \dot{\gamma}(t)\right\rangle|_{t=0}+\langle\nabla \rho(\gamma_{x,v}(t)), \nabla_{\dot{\gamma}(t)}\dot{\gamma}(t)\rangle |_{t=0}\\ &=\langle \nabla_v\nu(x),v \rangle +\langle \nu, \lambda(x,v)iv \rangle\\
  &=-\sff(x,v)+\langle \nu, \lambda(x,v)iv \rangle.
\end{align*}
This implies that $\partial M$ is strictly $\lambda$-convex if and only if $\left.  \frac{d^2}{d t^2}[\rho \circ \gamma_{x,v}(t)]\right|_{t=0}<0$. By Taylor's theorem, we have 
\begin{align*}
    \rho(\gamma_{x,v}(t))&=\rho(\gamma_{x,v}(0))+\left\langle \nabla \rho \left(x,v\right), \dot{\gamma}_{x, v}\left(0\right)\right\rangle t+\frac{1}{2}  \left.  \frac{d^2}{d t^2}[\rho \circ \gamma_{x,v}(t)]\right|_{t=0} t^2+O\left(t^3\right)\\ &=\rho(\gamma_{x,v}(0))+\left\langle\nu\left(x\right), v\right\rangle t+\frac{1}{2} \left.  \frac{d^2}{d t^2}[\rho \circ \gamma_{x,v}(t)]\right|_{t=0}t^2+O\left(t^3\right)\\
    &=\frac{1}{2} \left.  \frac{d^2}{d t^2}[\rho \circ \gamma_{x,v}(t)]\right|_{t=0}t^2+O\left(t^3\right),
\end{align*}
which is negative for small $|t|$ since $\left.  \frac{d^2}{d t^2}[\rho \circ \gamma_{x,v}(t)]\right|_{t=0}<0$. This shows that for small positive and negative times $\rho(\gamma(t))<0$, i.e., $\gamma_{x,v}(t)\in N\setminus M.$
\end{proof}
In the next subsection, we will see how the signed $\lambda$-curvature for $\lambda$-geodesics is related to its dual $\lambda$-geodesic.

\subsection{Curvature of the dual \texorpdfstring{$\lambda$}{lambda}-geodesic flow}
\label{subsec:dualflows-curva} 
Let us define the \emph{reversion map} $r: SM \rightarrow SM$ by $r(x, v) = (x, -v)$. Let $u \in C^{\infty}(SM)$, and consider its extension as a homogeneous function of degree zero in $C^{\infty}(TM \setminus {0})$, denoted by $u(x, y/|y|)$. We define the horizontal and vertical derivatives as follows:\begin{align*}
\nabla_{x_i} u & =\frac{\partial}{\partial x_i}(u(x, y /|y|))-\left.\Gamma_{i k}^l v^k \partial_{v_l} u\right|_{S M}, \\
\partial_{v_i} u & =\left.\frac{\partial}{\partial y_i}(u(x, y /|y|))\right|_{S M} .
\end{align*}
Here, $\Gamma_{i k}^l$ denotes the Christoffel symbols of the metric $g$ (cf. \cite[p. 386]{Ilmavirta:Salo:2016}).
\begin{lemma}\label{eq:chain-for-antipodal-map}
Let $(M, g)$ be a Riemannian surface with or without boundary. For any function $f \in C^1(SM)$, the following formulas hold:
\begin{enumerate}[(i)]
    \item $V(f \circ r)(x,v) = ((Vf) \circ r)(x,v)$.\label{eq:Vr-composition}
\item  $X(f \circ r)(x,v) = -((Xf) \circ r)(x,v)$.\label{eq:Xr-composition}
\item $X_\bot(f \circ r)(x,v) = -((X_\bot f) \label{eq:Xbotr-composition}\circ r)(x,v)$.
\end{enumerate}
\end{lemma}
\noindent
We omit the derivations of the formulas in \ref{eq:chain-for-antipodal-map}. Their proofs are straightforward computations.

Note that $\lambda^-=-\lambda \circ r$ by the definitions. We define $K_{\lambda^-} (x,v)$ as the Gaussian $\lambda^{-}$ curvature corresponding to the $\lambda^{-}$-geodesic flow. We prove in the next corollary that the global negative curvature assumption for the $\lambda$-geodesic flow is equivalent to the same property of the $\lambda^{-}$-geodesic flow. 
\begin{corollary}\label{cor:dual-curvatures}
Let $(M, g)$ be a Riemannian surface with or without boundary and $\lambda \in C^{\infty}(SM)$. Then the (Gaussian) $\lambda$-curvature of the dual system satisfies  
\[K_{\lambda^{-}}(x,v) = K_{\lambda}(x,-v), \qquad \text{ for all } (x, v) \in  SM.\] 
Additionally, if $\partial M\neq \emptyset$, then the signed $\lambda$-curvature of the dual system satisfies 
\[\kappa_{\lambda^{-}}(x,v) = \kappa_{\lambda}(x,-v),\qquad \text{ for all } (x, v) \in \partial SM.\]
\end{corollary}
\begin{proof}
Let $(x,v) \in SM$. We may compute using Lemma \ref{eq:chain-for-antipodal-map} that
\begin{align*}
    &K_{\lambda^-} (x,v)\\
    &\quad= K(x)+X_{\perp}(\lambda^-)(x,v)+ (\lambda^- )^2(x,v)+(X+\lambda^- V )V(\lambda^-)(x,v)\\
    &\quad=K(x)-X_{\perp}(\lambda\circ r)(x,v)+ \lambda^2(x,-v)-(X-(\lambda\circ r) V )V(\lambda\circ r)(x,v)\\
    &\quad=K(x)+X_{\perp}(\lambda)(x,-v)+\lambda^2(x,-v) -(X-\lambda(x,-v)V )V(\lambda)(x,-v)\\
    &\quad=K(x)+X_{\perp}(\lambda)(x,-v)+\lambda^2(x,-v) -X(V(\lambda)\circ r)(x,v)+\lambda(x,-v)V )V(\lambda)(x,-v)\\
    &\quad=K(x)+X_{\perp}(\lambda)(x,-v)+\lambda^2(x,-v)+ X(V(\lambda) )(x,-v)+\lambda (x,-v)V V(\lambda)(x,-v)\\
    &\quad= K(x)+X_{\perp}(\lambda)(x,-v)+\lambda^2(x,-v)+(X+\lambda(x,-v)V )V(\lambda)(x,-v)\\
   &\quad=K_{\lambda}(x,-v).
\end{align*}
By Lemma \ref{eq:chain-for-antipodal-map}, we obtain
\begin{align*}
  \kappa_{\lambda^-} (x,v) &=\kappa(x)-\langle\nu(x), \lambda^-(x, v) i v\rangle \\
  &= \kappa(x)+\langle\nu(x), \lambda(x, -v) i v\rangle\\
  &= \kappa(x)-\langle\nu(x), \lambda(x, -v) i (-v)\rangle\\
  &=\kappa_{\lambda}(x,-v). \qedhere
\end{align*}
\end{proof}
\begin{remark} \label{rmk:eta-dual-law}
In a similar manner we can see that 
$\eta_{\lambda^{-}}(x, v) = \eta_\lambda(x, -v)$. 
\begin{align*}
    \eta_{\lambda^{-}}(x, v)&= \langle V(\lambda^{-})(x, v) v, \nu\rangle=-\langle V(\lambda)(x, -v) v, \nu\rangle=\eta_\lambda(x, -v).
\end{align*}
\end{remark}

\subsection{Proof of Lemma {\ref{lm:property:rho:kappa:eta}}}\label{sec:auxiliary-proofs-appendix}
\ref{eq:kappa:circ:1}
We have
\begin{align*}
    \kappa_{\lambda}(x,v)&=\kappa(x)-\langle\nu(x), \lambda(x, v) i v\rangle,\\
    \kappa_{\lambda}\circ \rho(x,v)
    &=\kappa(x)-\langle\nu(x), \lambda\circ \rho(x, v) i (v-2\langle v,\nu(x) \rangle\nu(x))\rangle\\
    &=\kappa(x)-\langle\nu(x), \lambda\circ \rho(x, v)  (iv-2\langle v,\nu(x) \rangle i\nu(x))\rangle\\
    &=\kappa(x)-\langle\nu(x), \lambda\circ \rho(x, v)  iv\rangle\\
&=\kappa_{\lambda\circ \rho}(x,v).
\end{align*}
\ref{eq:eta:circ:1} Let us compute \begin{align*}
    \eta_\lambda(x, v)&=\langle V(\lambda)(x, v) v, \nu\rangle,\\
 \eta_\lambda\circ \rho (x, v)   &=\left\langle (V(\lambda)\circ \rho (x, v)) (v-2\langle v,\nu\rangle\nu ), \nu\right\rangle\\
 &=\langle -(V(\lambda\circ \rho) (x, v)) (v-2\langle v,\nu\rangle\nu ), \nu\rangle\\
 &=\langle V(\lambda\circ \rho)(x, v) v, \nu\rangle\\
 &=\eta_{\lambda\circ \rho}(x, v).
\end{align*}
\ref{eq:kappa:circ:2}
By the definition of even function, we have
\begin{align*}
(\kappa_\lambda)_e&=\frac{\kappa_\lambda+\kappa_\lambda\circ \rho}{2}\\
    &=\frac{\kappa_\lambda+\kappa_{\lambda\circ \rho}}{2}\\
    &= \kappa-\left\langle\nu(x), \frac{(\lambda+\lambda \circ \rho)(x, v)}{2} i v\right\rangle\\
    &=\kappa_{{\lambda}_{e}}.
\end{align*}
Similarly, one could get $  \eta_{\lambda_e}=(\eta_\lambda)_e$.

\noindent \ref{eq:eta:circ:2} This part directly follows form \ref{eq:kappa:circ:2}. In particular, we have
\begin{align*}
     \rho^*(\kappa_{\lambda_e})&=\kappa_{\rho^*\lambda_e}=\kappa_{\lambda_e},
\end{align*}
and 
\begin{align*}
     \rho^*(\eta_{\lambda_e})&=\eta_{\rho^*\lambda_e}=\eta_{\lambda_e},
\end{align*}
since $\rho^*\lambda_e=\lambda_e$.
\qed
\subsection*{Acknowledgments} M.K. would like to thank Mikko Salo for suggesting research on the magnetic broken ray transforms and helpful discussions. J.R. thanks Gabriel P. Paternain for many helpful discussions related to this work. We thank an anonymous referee for many helpful comments that improved the final manuscript. S.R.J. and J.R. would like to thank the Isaac Newton Institute for Mathematical Sciences, Cambridge, UK, for support and hospitality during \emph{Rich and Nonlinear Tomography - a multidisciplinary approach} in 2023 where part of this work was done (supported by EPSRC Grant Number EP/R014604/1). The work of S.R.J. and J.R. was supported by the Research Council of Finland through the Flagship of Advanced Mathematics for Sensing, Imaging and Modelling (decision number 359183). S.R.J. acknowledges the Prime Minister's Research Fellowship (PMRF) from the Government of India for his PhD work.  M.K. was supported by MATRICS grant (MTR/2019/001349) of SERB. J.R. was supported by the Vilho, Yrjö and Kalle Väisälä Foundation of the Finnish Academy of Science and Letters.
\subsection*{Data availability statement} Data sharing not applicable to this article as no datasets were generated or analyzed during the current study.

\bibliography{math} 
	
\bibliographystyle{alpha}
\end{document}